\documentclass[12pt,reqno]{amsart}
\usepackage{amssymb,amscd,url}

\usepackage[all]{xy}

\begin{document}

%% Eliminates frameboxes because ArXiv had trouble with them
%% \def\framebox#1{\text{\textbf{[\negthinspace[}#1\textbf{]\negthinspace]}}} 

\allowdisplaybreaks

%%%%%%%%%%%%%%%%%%%%%%%%%%%%%%%%%%%%%%%%%%%%%%%%%%%%%%%%%%%%%%%%%%%%%%
%% Title and Author Information

\title[Integrality properties of B\"ottcher coordinates]
      {Integrality properties of B\"ottcher coordinates for one-dimensional superattracting germs}
\date{\today}
\author{Adriana Salerno}
\email{asalerno@bates.edu}
%% Adriana Salerno <adriana.salerno@gmail.com>
\address{Department of Mathematics, Bates College, Lewiston, ME 04240 USA}
\author[Joseph H. Silverman]{Joseph H. Silverman}
\email{jhs@math.brown.edu}
\address{Mathematics Department, Box 1917
  Brown University, Providence, RI 02912 USA.
  ORCID: https://orcid.org/0000-0003-3887-3248}
\thanks{Silverman's research supported by Simons Collaboration Grant \#241309}

%%%%%%%%%%%%%%%%%%%%%%%%%%%%%%%%%%%%%%%%%%%%%%%%%%%%%%%%%%%%%%%%%%%%%%

% \allowdisplaybreaks

\hyphenation{ca-non-i-cal semi-abel-ian}

%%%%%%%%%%%%%%%%%%%%%%%%%%%%%%%%%%%%%%%%%%%%%%%%%%%%%%%%%%%%%%%%%%%%%%
% Theorem environments

\newtheorem{theorem}{Theorem}
\newtheorem{lemma}[theorem]{Lemma}
\newtheorem{sublemma}[theorem]{Sublemma}
\newtheorem{conjecture}[theorem]{Conjecture}
\newtheorem{proposition}[theorem]{Proposition}
\newtheorem{corollary}[theorem]{Corollary}
\newtheorem*{claim}{Claim}

\theoremstyle{definition}
% The * surpresses numbering
\newtheorem*{definition}{Definition}
\newtheorem{example}[theorem]{Example}
\newtheorem{remark}[theorem]{Remark}
\newtheorem{question}[theorem]{Question}

\theoremstyle{remark}
\newtheorem*{acknowledgement}{Acknowledgements}

%%%%%%%%%%%%%%%%%%%%%%%%%%%%%%%%%%%%%%%%%%%%%%%%%%%%%%%%%%%%%%%%%%%%%%

%%%%%%%% Set Up Environment for Notation %%%%%%%%%%%%%%
% This is currently set to allow quite wide items to be defined
\newenvironment{notation}[0]{%
  \begin{list}%
    {}%
    {\setlength{\itemindent}{0pt}
     \setlength{\labelwidth}{4\parindent}
     \setlength{\labelsep}{\parindent}
     \setlength{\leftmargin}{5\parindent}
     \setlength{\itemsep}{0pt}
     }%
   }%
  {\end{list}}

%%%%%%%% Set Up Environment for Parts in Theorems %%%%%%%%%%%%%%
\newenvironment{parts}[0]{%
  \begin{list}{}%
    {\setlength{\itemindent}{0pt}
     \setlength{\labelwidth}{1.5\parindent}
     \setlength{\labelsep}{.5\parindent}
     \setlength{\leftmargin}{2\parindent}
     \setlength{\itemsep}{0pt}
     }%
   }%
  {\end{list}}
% Use \Part{(a)}, instead of \item[(a)], to ensure upright font
\newcommand{\Part}[1]{\item[\upshape#1]}

%%%%%%%% Set Up Macro for Cases %%%%%%%%%%%%%%
\def\Case#1#2{%
\paragraph{\textbf{\boldmath Case #1: #2.}}\hfil\break\ignorespaces}

%%%%%%%%%%%%%%%%%%
% Greek Alphabet %
%%%%%%%%%%%%%%%%%%
\renewcommand{\a}{\alpha}
\renewcommand{\b}{\beta}
\newcommand{\bfbeta}{{\boldsymbol{\beta}}}
\newcommand{\g}{\gamma}
\renewcommand{\d}{\delta}
\newcommand{\e}{\epsilon}
\newcommand{\f}{\varphi}
\newcommand{\bfphi}{{\boldsymbol{\f}}}
\renewcommand{\l}{\lambda}
\renewcommand{\k}{\kappa}
\newcommand{\lhat}{\hat\lambda}
\newcommand{\m}{\mu}
\newcommand{\bfmu}{{\boldsymbol{\mu}}}
\renewcommand{\o}{\omega}
\renewcommand{\r}{\rho}
\newcommand{\rbar}{{\bar\rho}}
\newcommand{\s}{\sigma}
\newcommand{\sbar}{{\bar\sigma}}
\renewcommand{\t}{\tau}
\newcommand{\z}{\zeta}

\newcommand{\D}{\Delta}
\newcommand{\G}{\Gamma}
\newcommand{\F}{\Phi}
\renewcommand{\L}{\Lambda}

%%%%%%%%%%%%%%%%%%%%
% Fraktur Alphabet %
%%%%%%%%%%%%%%%%%%%%
\newcommand{\ga}{{\mathfrak{a}}}
\newcommand{\gb}{{\mathfrak{b}}}
\newcommand{\gn}{{\mathfrak{n}}}
\newcommand{\gp}{{\mathfrak{p}}}
\newcommand{\gP}{{\mathfrak{P}}}
\newcommand{\gq}{{\mathfrak{q}}}

%%%%%%%%%%%%%%%%%%%
% Barred Alphabet %
%%%%%%%%%%%%%%%%%%%
\newcommand{\Abar}{{\bar A}}
\newcommand{\Ebar}{{\bar E}}
\newcommand{\kbar}{{\bar k}}
\newcommand{\Kbar}{{\bar K}}
\newcommand{\Pbar}{{\bar P}}
\newcommand{\Sbar}{{\bar S}}
\newcommand{\Tbar}{{\bar T}}
\newcommand{\gbar}{{\bar\gamma}}
\newcommand{\lbar}{{\bar\lambda}}
\newcommand{\ybar}{{\bar y}}
\newcommand{\phibar}{{\bar\f}}

%%%%%%%%%%%%%%%%%%%%%%%%%
% Calligraphic Alphabet %
%%%%%%%%%%%%%%%%%%%%%%%%%
\newcommand{\Acal}{{\mathcal A}}
\newcommand{\Bcal}{{\mathcal B}}
\newcommand{\Ccal}{{\mathcal C}}
\newcommand{\Dcal}{{\mathcal D}}
\newcommand{\Ecal}{{\mathcal E}}
\newcommand{\Fcal}{{\mathcal F}}
\newcommand{\Gcal}{{\mathcal G}}
\newcommand{\Hcal}{{\mathcal H}}
\newcommand{\Ical}{{\mathcal I}}
\newcommand{\Jcal}{{\mathcal J}}
\newcommand{\Kcal}{{\mathcal K}}
\newcommand{\Lcal}{{\mathcal L}}
\newcommand{\Mcal}{{\mathcal M}}
\newcommand{\Ncal}{{\mathcal N}}
\newcommand{\Ocal}{{\mathcal O}}
\newcommand{\Pcal}{{\mathcal P}}
\newcommand{\Qcal}{{\mathcal Q}}
\newcommand{\Rcal}{{\mathcal R}}
\newcommand{\Scal}{{\mathcal S}}
\newcommand{\Tcal}{{\mathcal T}}
\newcommand{\Ucal}{{\mathcal U}}
\newcommand{\Vcal}{{\mathcal V}}
\newcommand{\Wcal}{{\mathcal W}}
\newcommand{\Xcal}{{\mathcal X}}
\newcommand{\Ycal}{{\mathcal Y}}
\newcommand{\Zcal}{{\mathcal Z}}

%%%%%%%%%%%%%%%%%%%%%%%%%%%%
% Blackboard Bold Alphabet %
%%%%%%%%%%%%%%%%%%%%%%%%%%%%
\renewcommand{\AA}{\mathbb{A}}
\newcommand{\BB}{\mathbb{B}}
\newcommand{\CC}{\mathbb{C}}
\newcommand{\FF}{\mathbb{F}}
\newcommand{\GG}{\mathbb{G}}
\newcommand{\NN}{\mathbb{N}}
\newcommand{\PP}{\mathbb{P}}
\newcommand{\QQ}{\mathbb{Q}}
\newcommand{\RR}{\mathbb{R}}
\newcommand{\ZZ}{\mathbb{Z}}

%%%%%%%%%%%%%%%%%%%%%%%%%%
% Boldface Math Alphabet %
%%%%%%%%%%%%%%%%%%%%%%%%%%
\newcommand{\bfa}{{\boldsymbol a}}
\newcommand{\bfb}{{\boldsymbol b}}
\newcommand{\bfc}{{\boldsymbol c}}
\newcommand{\bfd}{{\boldsymbol d}}
\newcommand{\bfe}{{\boldsymbol e}}
\newcommand{\bff}{{\boldsymbol f}}
\newcommand{\bfg}{{\boldsymbol g}}
\newcommand{\bfi}{{\boldsymbol i}}
\newcommand{\bfj}{{\boldsymbol j}}
\newcommand{\bfn}{{\boldsymbol n}}
\newcommand{\bfp}{{\boldsymbol p}}
\newcommand{\bfr}{{\boldsymbol r}}
\newcommand{\bfs}{{\boldsymbol s}}
\newcommand{\bft}{{\boldsymbol t}}
\newcommand{\bfu}{{\boldsymbol u}}
\newcommand{\bfv}{{\boldsymbol v}}
\newcommand{\bfw}{{\boldsymbol w}}
\newcommand{\bfx}{{\boldsymbol x}}
\newcommand{\bfy}{{\boldsymbol y}}
\newcommand{\bfz}{{\boldsymbol z}}
\newcommand{\bfA}{{\boldsymbol A}}
\newcommand{\bfF}{{\boldsymbol F}}
\newcommand{\bfB}{{\boldsymbol B}}
\newcommand{\bfD}{{\boldsymbol D}}
\newcommand{\bfG}{{\boldsymbol G}}
\newcommand{\bfI}{{\boldsymbol I}}
\newcommand{\bfM}{{\boldsymbol M}}
\newcommand{\bfP}{{\boldsymbol P}}
\newcommand{\bfzero}{{\boldsymbol{0}}}
\newcommand{\bfone}{{\boldsymbol{1}}}

%%%%%%%%%%%%%%%%%%%%%%%%%%%%%%
% Miscellaneous New Commands %
%%%%%%%%%%%%%%%%%%%%%%%%%%%%%%
\newcommand{\Aut}{\operatorname{Aut}}
\newcommand{\Berk}{{\textup{Berk}}}
\newcommand{\Birat}{\operatorname{Birat}}
\newcommand{\codim}{\operatorname{codim}}
\newcommand{\Crit}{\operatorname{Crit}}
\newcommand{\critwt}{\operatorname{critwt}} % valency of a portrait
\newcommand{\diag}{\operatorname{diag}}
\newcommand{\Disc}{\operatorname{Disc}}
\newcommand{\Div}{\operatorname{Div}}
\newcommand{\Dom}{\operatorname{Dom}}
\newcommand{\End}{\operatorname{End}}
\newcommand{\Fbar}{{\bar{F}}}
\newcommand{\Fix}{\operatorname{Fix}}
\newcommand{\Gal}{\operatorname{Gal}}
\newcommand{\GITQuot}{/\!/}
\newcommand{\GL}{\operatorname{GL}}
\newcommand{\GR}{\operatorname{\mathcal{G\!R}}}
\newcommand{\Hom}{\operatorname{Hom}}
\newcommand{\Index}{\operatorname{Index}}
\newcommand{\Image}{\operatorname{Image}}
\newcommand{\Isom}{\operatorname{Isom}}
\newcommand{\hhat}{{\hat h}}
\newcommand{\Ker}{{\operatorname{ker}}}
\newcommand{\Lift}{\operatorname{Lift}}
\newcommand{\limstar}{\lim\nolimits^*}
\newcommand{\limstarn}{\lim_{\hidewidth n\to\infty\hidewidth}{\!}^*{\,}}
\newcommand{\Mat}{\operatorname{Mat}}
\newcommand{\maxplus}{\operatornamewithlimits{\textup{max}^{\scriptscriptstyle+}}}
\newcommand{\MOD}[1]{~(\textup{mod}~#1)}
\newcommand{\Mor}{\operatorname{Mor}}
\newcommand{\Moduli}{\mathcal{M}}
\newcommand{\Norm}{{\operatorname{\mathsf{N}}}}
\newcommand{\notdivide}{\nmid}
\newcommand{\normalsubgroup}{\triangleleft}
\newcommand{\NS}{\operatorname{NS}}
\newcommand{\onto}{\twoheadrightarrow}
\newcommand{\ord}{\operatorname{ord}}
\newcommand{\Orbit}{\mathcal{O}}
\newcommand{\Pcase}[3]{\par\noindent\framebox{$\boldsymbol{\Pcal_{#1,#2}}$}\enspace\ignorespaces}
\newcommand{\Per}{\operatorname{Per}}
\newcommand{\Perp}{\operatorname{Perp}}
\newcommand{\PrePer}{\operatorname{PrePer}}
\newcommand{\PGL}{\operatorname{PGL}}
\newcommand{\Pic}{\operatorname{Pic}}
\newcommand{\Prob}{\operatorname{Prob}}
\newcommand{\Proj}{\operatorname{Proj}}
\newcommand{\Qbar}{{\bar{\QQ}}}
\newcommand{\rank}{\operatorname{rank}}
\newcommand{\Rat}{\operatorname{Rat}}
\newcommand{\Resultant}{\operatorname{Res}}
\renewcommand{\setminus}{\smallsetminus}
\newcommand{\sgn}{\operatorname{sgn}} 
\newcommand{\shafdim}{\operatorname{ShafDim}}
\newcommand{\SL}{\operatorname{SL}}
\newcommand{\Span}{\operatorname{Span}}
\newcommand{\Spec}{\operatorname{Spec}}
\renewcommand{\ss}{\textup{ss}}
\newcommand{\stab}{\textup{stab}}
\newcommand{\Stab}{\operatorname{Stab}}
\newcommand{\Support}{\operatorname{Supp}}
\newcommand{\TableLoopSpacing}{{\vrule height 15pt depth 10pt width 0pt}} %% Put extra space into table with loop figures
\newcommand{\tors}{{\textup{tors}}}
\newcommand{\Trace}{\operatorname{Trace}}
\newcommand{\trianglebin}{\mathbin{\triangle}} % symmetric set difference
\newcommand{\tr}{{\textup{tr}}} % for K/k trace
\newcommand{\UHP}{{\mathfrak{h}}}    % Upper half plane
\newcommand{\val}{\operatorname{val}} % valency of a portrait
\newcommand{\wt}{\operatorname{wt}} %% weight of a portrait
\newcommand{\<}{\langle}
\renewcommand{\>}{\rangle}

\newcommand{\pmodintext}[1]{~\textup{(mod}~#1\textup{)}}
\newcommand{\ds}{\displaystyle}
\newcommand{\longhookrightarrow}{\lhook\joinrel\longrightarrow}
\newcommand{\longonto}{\relbar\joinrel\twoheadrightarrow}
\newcommand{\SmallMatrix}[1]{%
  \left(\begin{smallmatrix} #1 \end{smallmatrix}\right)}

\newcommand{\SD}{\operatorname{\mathcal{SD}}}  %% abbreviation for ShafDim
\newcommand{\MD}{\operatorname{\mathcal{MD}}}  %% abbreviation for dim \Moduli

%%%%%%%%%%%%%%%%%%%%%%%%%%%%%%%%%%%%%%%%%%%%%%%%%%%%%%%%%%%%%%%%%%%%%%

\begin{abstract}
Let $R$ be a ring of characteristic $0$ with field of fractions $K$, and let $m\ge2$.  The B\"ottcher coordinate of a power series $\varphi(x)\in x^m + x^{m+1}R[\![x]\!]$ is the unique power series $f_\varphi(x)\in x+x^2K[\![x]\!]$ satisfying $\varphi\circ f_\varphi(x) = f_\varphi(x^m)$.  In this paper we study the integrality properties of the coefficients of $f_\varphi(x)$, partly for their intrinsic interest and partly for potential applications to $p$-adic dynamics. Results include: (1) If $p$ is prime and $R=\mathbb Z_p$ and $\varphi(x)\in x^p + px^{p+1}R[\![x]\!]$, then $f_\varphi(x)\in R[\![x]\!]$.  (2) If $\varphi(x)\in x^m + mx^{m+1}R[\![x]\!]$, then $f_\varphi(x)=x\sum_{k=0}^\infty a_kx^k/k!$ with all $a_k\in R$.  (3) In (2), if $m=p^2$, then $a_k\equiv-1\pmodintext{p}$ for all $k$ that are powers of $p$.
\end{abstract}

%% Non-TeX abstract

%% Let R be a ring of characteristic 0 with field of fractions K, and let m ≥ 2.  The B\"ottcher coordinate of a power series φ(x) ∈ x<SUP>m</SUP> + x<SUP>m+1</SUP>R[[x]] is the unique power series f<sub>φ</sub>(x) ∈ x+x<SUP>2</SUP>K[[x]] satisfying φ(f<sub>φ</sub>(x)) = f<sub>φ</sub>(x<SUP>m</SUP>).  In this paper we study the integrality properties of the coefficients of f<sub>φ</sub>(x), partly for their intrinsic interest and partly for potential applications to p-adic dynamics. Results include: (1) If p is prime and φ(x) ∈ x<SUP>p</SUP> + px<SUP>p+1</SUP>R[[x]], then f<sub>φ</sub>(x) ∈ R[[x]]. (2) If φ(x) ∈ x<SUP>m</SUP> + mx<SUP>m+1</SUP>R[[x]], then f<sub>φ</sub>(x)=x ∑<sub>k≥0</sub> a_kx<SUP>k</SUP>/k! with all a_k ∈ R.  (3) In (2), if m=p<sup>2</sup>, then a<sub>k</sub> = -1 (mod p) for all k that are powers of p.

\maketitle

%% \tableofcontents

%%%%%%%%%%%%%%%%%%%%%%%%%%%%%%%%%%%%%%%%%%%%%%%%%%%%%%%%%%%%%%%%%%%%%%
\section{Introduction}
\label{section:introduction}
%%%%%%%%%%%%%%%%%%%%%%%%%%%%%%%%%%%%%%%%%%%%%%%%%%%%%%%%%%%%%%%%%%%%%%

The following well-known result is essentially due to
B{\"o}ttcher~\cite{Bottcher}.

\begin{proposition}
\label{proposition:fphidef}
Let~$K$ be a field of characteristic~$0$, and let $m\ge2$.
Let
\[
  \f(x) \in x^m + x^{m+1}K[\![x]\!]
\]
be a power series of the indicated form. Then there is a unique formal power
series $f_\f(x)\in x+x^2K[\![x]\!]$ satisfying
\begin{equation}
  \label{eqn:bottdef}
  \f\circ f_\f(x) = f_\f(x^m).
\end{equation}
\end{proposition}

There are two standard ways to prove
Proposition~\ref{proposition:fphidef}.  First, one can use the
B\"ottcher equation~\eqref{eqn:bottdef} to construct a recursion that
defines each coefficient of~$f_\f(x)$ in terms of the earlier
coefficients and the coefficients of~$\f(x)$. Second, one can show
that
\begin{equation}
\label{eqn:bottaslim}
  f_\f(x) := \lim_{n\to\infty} \bigl(\f^{\circ n}(x)\bigr)^{1/m^n}
  \quad\text{converges in $K[\![x]\!]$.}
\end{equation}

\begin{definition}
The series~$f_\f(x)\in K[\![x]\!]$ uniquely determined
by~\eqref{eqn:bottdef} is called the (\emph{local}) \emph{B{\"o}ttcher
  coordinate} for the series~$\f(x)$.
\end{definition}

For~$K=\CC$, B{\"o}ttcher proved that if~$\f(x)$ is analytic at~$0$,
then~$f_\f(x)\in\CC[\![x]\!]$ converges on a neighborhood of~$0$, and thus
gives a local complex analytic conjugacy between~$\f(x)$ and~$x^m$.
See~\cite[Chapter~9]{MR2193309}, for example, for a discussion of
B\"ottcher coordinates over~$\CC$.  In this paper we are interested in
the convergence properties of the series~$f_\f$ in the case that~$K$
is a non-archimedean field, or alternatively, we want to study the
integrality properties of the coefficients of the B\"ottcher
coordinate.  Suppose that~$\f(x)\in R[\![x]\!]$ has coefficients in a
ring~$R$.  If we further assume that the ramification degree~$m$ is
invertible in~$R$, then the coefficients of~$f_\f(x)$ are
quite well-behaved, as in the following result.

\begin{proposition}
\label{proposition:botmprimetop}
Let~$R$ be a ring,
let~$m\ge2$ be an integer satisfying~$m\in R^*$, and let
\begin{equation}
  \label{eqn:phixxmxm1Rxx}
  \f(x) \in x^m +  x^{m+1}R[\![x]\!]
\end{equation}
be a power series of the indicated form.  Then both the B\"ottcher
coordinate~$f_\f(x)$ and its inverse series~$f_\f^{-1}(x)$ are in
$R[\![x]\!]$.
\end{proposition}
\begin{proof}
This is well-known, cf.\ \cite{MR3064415}. It follows easily via
an induction argument similar to the proof of Theorem~\ref{theorem:compm}.
\end{proof}

The coefficients of the B\"ottcher coordinate become more complicated,
and much more interesting, when the ramification degree~$m$ is not
invertible in~$R$.  Our first main result gives a general bound for
the denominators of the coefficients of the B\"ottcher coordinate for
maps of the form~\eqref{eqn:phixxmxm1Rxx} without the assumption
that~$m$ is invertible in~$R$. We also give a better bound if a few of
the non-leading coefficients of~$\f(x)$ have some additional
$m$-divisibility.

\begin{theorem}
\label{theorem:compm}
Let~$R$ be a ring of characteristic~$0$,
let~$m\ge2$ be an integer, and let
\[
  \f(x) = x^m \sum_{k=0}^\infty b_kx^k \in R[\![x]\!]\quad\text{with $b_0=1$.}
\]
\vspace{-10pt}  
\begin{parts}
\Part{(a)}
Both the B\"ottcher coordinate~$f_\f(x)$ and its
inverse~$f_\f^{-1}(x)$ are series of the form
\[
  x \sum_{k=0}^\infty \frac{a_k}{m^k k!}x^k \quad\text{with $a_0=1$ and $a_k\in R$ for all $k$.}
\]
\Part{(b)}  
Suppose further that the coefficients of~$\f$ satisfy
\[
  k!b_k \in m R\quad\text{for $1\le k< m$.}
\]
For example, this is true if $\f(x)\in x^m+mx^{m+1}R[\![x]\!]$.
Then the B\"ottcher coordinate~$f_\f(x)$ and its
inverse~$f_\f^{-1}(x)$ are series of the form
\[
  x\sum_{k=0}^\infty \frac{a_k}{k!}x^k \quad\text{with $a_0=1$ and $a_k\in R$ for all $k$.}
\]
\end{parts}
\end{theorem}

A special case of Theorem~\ref{theorem:compm}(b) says that if
$\f(x)\in x^m + mx^{m+1}R[\![x]\!]$, then the B\"ottcher
coordinate~$f_\f(x)$ has the form $x\sum a_kx^k/k!$ with $a_k\in R$.
It turns out that if~$m$ is prime, then we can often do much better,
as shown in the following somewhat surprising result.

\begin{theorem}
\label{theorem:fphiinRp}
Let~$R$ be a ring of characteristic~$0$ with fraction field~$K$, and
let~$p$ be a prime such that $a^p\equiv a\pmodintext{pR}$ for
all~$a\in R$. For example,~$R$ could be~$\ZZ$ or~$\ZZ_p$. Let
\[
  \f(x) \in x^p + p x^{p+1}R[\![x]\!]
\]
be a power series of the indicated form. Then the B\"ottcher
coordinate and its inverse satisfy
\[
  f_\f(x)\in R[\![x]\!]\quad\text{and}\quad f_\f^{-1}(x)\in R[\![x]\!].
\]
\end{theorem}

We've already noted that Theorem~\ref{theorem:fphiinRp}, which deals
with the case that $m=p$ is prime, is much stronger than
Theorem~\ref{theorem:compm}(b), which deals with the case that~$m$ is
composite.  The proof of Theorem~\ref{theorem:fphiinRp} relies on
Fermat's little theorem, so one might suppose that
Theorem~\ref{theorem:compm}(b) could be strengthened by using
the congruence
\[
  (a+b)^m \cong (a^p+b^p)^{m/p} \pmod{pR},
\]
which valid for $p\mid m$. However, this is not the case, as shown by the
following result, whose proof Section~\ref{section:nonint} is a
complicated induction on the coefficients of the B\"ottcher
coordinate.

\begin{theorem}
\label{theorem:radiusp2}
Let $\f(x)=x^{p^2}+p^2 x^{p^2+1}$, let~$f_\f(x)$ be the B\"ottcher
coordinate for~$\f$, and write~$f_\f(x)$ as
\[
  f_\f(x) = x\sum_{k=0}^\infty \frac{a_k}{k!}x^k,
\]
where~$a_k\in\ZZ$ from Theorem~$\ref{theorem:compm}(b)$.
Then for all~$k$ that are powers of~$p$, we have
\[
  a_k \equiv -1 \pmod{p}.
%%  a_k \not\equiv 0\pmod{p}.
\]
\end{theorem}

Table~\ref{table:botexample} illustrates our results by giving the
first few terms of the B\"ottcher coordinate of $\f(x)\in x^m+
mx^{m+1}R[\![x]\!]$ for small values of~$m$.

\begin{table}
  %\small
  \tiny
\[
\begin{array}{|c|l|}\hline
  m & \text{B\"ottcher coordinate of $x^m+mx^{m+1}$} \\ \hline\hline
  2 & \vrule height10pt depth5pt width0pt
  x - x^2 + 2 x^3 - 7 x^4 + 26 x^5 - 98 x^6 + 389 x^7 - 1617 x^8 +
  6884 x^9 +\cdots \\ \hline
  3 & \vrule height10pt depth5pt width0pt
  x - x^2 + 3 x^3 - 12 x^4 + 52 x^5 - 246 x^6 + 1224 x^7 - 6300 x^8 +
  33300 x^9 +\cdots \\ \hline
  4 & \vrule height10pt depth5pt width0pt
  x - x^2 + \tfrac{7}{2} x^3 - 16 x^4 + \tfrac{661}{8} x^5 -
  \tfrac{923}{2} x^6 + \tfrac{43221}{16} x^7 - 16368 x^8 +
  \tfrac{13029155}{128} x^9 +\cdots \\ \hline
  5 & \vrule height10pt depth5pt width0pt
  x - x^2 + 4 x^3 - 21 x^4 + 125 x^5 - 801 x^6 + 5386 x^7 - 37497
  x^8 + 267913 x^9 +\cdots \\ \hline
  6 & \vrule height10pt depth5pt width0pt
  x - x^2 + \tfrac{9}{2} x^3 - \tfrac{80}{3} x^4 + \tfrac{4301}{24}
  x^5 - 1296 x^6 + \tfrac{1416521}{144} x^7 - \tfrac{695549}{9} x^8 +
  \tfrac{79748667}{128} x^9 +\cdots \\ \hline
  7 & \vrule height10pt depth5pt width0pt
  x - x^2 + 5 x^3 - 33 x^4 + 247 x^5 - 1989 x^6 + 16807 x^7 -
  146968 x^8 + 1318564 x^9 +\cdots \\ \hline
  8 & \vrule height10pt depth5pt width0pt
  x - x^2 + \tfrac{11}{2} x^3 - 40 x^4 + \tfrac{2639}{8} x^5 - 2926
  x^6 + \tfrac{435643}{16} x^7 - 262144 x^8 + \tfrac{331406059}{128}
  x^9 +\cdots \\ \hline
  9 & \vrule height10pt depth5pt width0pt
  x - x^2 + 6 x^3 - \tfrac{143}{3} x^4 + \tfrac{1288}{3} x^5 - 4158
  x^6 + \tfrac{380120}{9} x^7 - \tfrac{3994133}{9} x^8 + 4782969 x^9 
  +\cdots \\ \hline
  10 & \vrule height10pt depth5pt width0pt
  x - x^2 + \tfrac{13}{2} x^3 - 56 x^4 + \tfrac{4375}{8} x^5 -
  \tfrac{28704}{5} x^6 + \tfrac{5055273}{80} x^7 - \tfrac{3596928}{5}
  x^8 + \tfrac{5375265623}{640} x^9 +\cdots \\ \hline
  %% 4 &  \vrule height10pt depth5pt width0pt
  %% x - x^2 + \tfrac{7}{2} x^3 - 16 x^4 + \tfrac{661}{8} x^5 - \tfrac{923}{2} x^6 + \tfrac{43221}{16} x^7 + \cdots \\ \hline
  %% 6 &  \vrule height10pt depth5pt width0pt
  %% x - x^2 + \tfrac{9}{2} x^3 - \tfrac{80}{3} x^4 + \tfrac{4301}{24} x^5 - 1296 x^6 + \tfrac{1416521}{144} x^7 + \cdots \\ \hline
\end{array}
\]
\caption{B\"ottcher coordinate of $x^m+mx^{m+1}$}
\label{table:botexample}
\end{table}

B\"ottcher coordinates of polynomials over $p$-adic fields have been
investigated in~\cite{arxiv1703.05365,MR3064415}, where they are
applied to the study of $p$-adic dynamics. (See
Section~\ref{section:earlierwork} for details.) In this context, a key
quantity is the radius of convergence of the B\"ottcher
coordinate. Our main results yield the following estimates for this
radius.

\begin{corollary}
\label{corollary:botradconv}  
Let~$p$ be a prime, let $R_p=\{c\in\CC_p : \|c\|_p\le1\}$ be the ring of
integers of~$\CC_p$, and let $m\ge2$ be an integer.  For each 
indicated type of map~$\f$ , the B\"ottcher coordinate~$f_\f$ and its
inverse~$f_\f^{-1}$ converge on the indicated disk~$\Dcal$ and define an isometry
\[
  f_\f : \Dcal \xrightarrow{\;\;\sim\;\;} \Dcal.
\]
In particular,~$\f(x)$ is $p$-adically analytically conjugate to~$x^m$
on~$\Dcal$.
\begin{parts}
\Part{(a)}
For $\f(x)\in x^m+x^{m+1}R_p[\![x]\!]$, as in
Proposition~$\ref{proposition:botmprimetop}$ and
Theorem~$\ref{theorem:compm}(a)$, we may take
\[
  \Dcal = \begin{cases}
    \bigl\{ x\in\CC_p : \|x\|_p < 1 \bigr\} &\text{if $p\nmid m$.}\\
    \bigl\{ x\in\CC_p : \|x\|_p < p^{-1/(p-1)}\|m\|_p \bigr\} &\text{if $p\mid m$.}\\
  \end{cases}
\]
\Part{(b)}
For $\f(x)=x^m\sum_{k=0}^\infty b_kx^k/k!$ with $b_0=1$ and $k!b_k\in mR_p[\![x]\!]$ for all
$1\le k<m$, as in Theorem~$\ref{theorem:compm}(b)$, we may take
\[
  \Dcal =   \bigl\{ x\in\CC_p : \|x\|_p < p^{-1/(p-1)} \bigr\}.
\]
\Part{(c)}
For $\f(x)\in x^p+px^{p+1}R_p[\![x]\!]$ as in
Theorem~\ref{theorem:fphiinRp}, we may take
\[
  \Dcal =   \bigl\{ x\in\CC_p : \|x\|_p < 1 \bigr\}.
\]
\Part{(d)}
For $\f(x)=x^{p^2}+p^2x^{p^2+1}$ as in Theorem~\ref{theorem:radiusp2},
the radius of convergence of the B\"ottcher coordinate~$f_\f(x)$
is exactly equal to $p^{-1/(p-1)}$.
\end{parts}
\end{corollary}

\begin{remark}
Theorem~\ref{theorem:radiusp2} says that the $k$'th
coefficient~$a_k/k!$ of the B\"ottcher coordinate
of~$x^{p^2}+p^2x^{p^2+1}$ satisfies $a_k\equiv-1\pmodintext{p}$
provided that~$k$ is a power of~$p$.  Experiments suggest that this
reflects a much more widespread phenomenon.  For example, we suspect
that if~$k$ is a multiple of~$p$, then we always have
\[
  a_k \equiv (-1)^{k/p}\pmodintext{p}.
\]
In Section~\ref{section:xp2conj} we assemble a number of conjectures,
based on numerical evidence, that describe various $p$-adic properties
of the B\"ottcher coordinate of $x^{p^2}+p^{r+2}x^{p^2+1}$. In
particular, we conjecture that the radius of convergence of the
B\"ottcher coordinate is exactly $p^{-p^{-r}/(p-1)}$.
\end{remark}

\begin{remark}
Theorem~\ref{theorem:fphiinRp} tells us that if~$t\in\ZZ_p$, then the
B\"ottcher coordinate of $\f(x)=x^p+ptx^{p+1}$ has $p$-integral
coefficients. However, if we treat~$t$ as an indeterminate, then the
coefficients of~$f_\f(x)$ are in~$\QQ[t]$, but they often fail to be
in~$\ZZ[t]$. For example, for~$p=2$ we find that
{\tiny
  \[
    f_\f(x) =
  x -t x^{2} + \left( \dfrac{5 t^2 - t}{2} \right) x^{3} - (8 t^3 -
  t^2) x^{4} + \left( \dfrac{231 t^4 - 30 t^3 + 9 t^2 - 2 t}{8} \right)
  x^{5} +\cdots\,,
  \]
}%
and for $p=3$ we have
{\tiny
  \[
  f_\f(x) = 
  x -t x^{2} + 3 t^2 x^{3} - \left( \dfrac{35 t^3 + t}{3} \right)
  x^{4} + \left( \dfrac{154 t^4 + 2 t^2}{3} \right) x^{5} - (243 t^5 + 3 t^3) x^{6} +\cdots\,.
  \]
}%
Although the coefficients of~$f_\f(x)$ are in~$\QQ[t]$, we can verify
that they are integer-valued, as they must be according to
Theorem~\ref{theorem:fphiinRp}, by writing their Newton--Mahler
expansions. For example, for~$p=2$ the coefficient of~$x^5$
in~$f_\f(x)$ has Newton--Mahler expansion
\[
  \dfrac{231 t^4 - 30 t^3 + 9 t^2 - 2 t}{8}=
  693 \binom{t}{4} + 1017 \binom{t}{3} + 384 \binom{t}{2} + 26 \binom{t}{1}.
\]
\end{remark}

\begin{remark}
In this paper we start with a power series~$\f(x)=x^m+\cdots$ having a
critical point at~$0$ and study the arithmetic properties of the
coefficients of the B\"ottcher coordinate~$f_\f(x)$ that
conjugates~$\f(x)$ to~$x^m$. We mention that if instead we start with
an invertible power series~$f(x)=x+\cdots$, then there is a unique
power series for which~$f$ is the $m$-power B\"ottcher coordinate.
Indeed, replacing~$x$ by~$f^{-1}(x)$ in the B\"ottcher
equation~\eqref{eqn:bottdef} yields
\[
  \f(x)=f\bigl(f^{-1}(x)^m\bigr),
\]
and this~$\f$ clearly satisfies $\f(x)=x^m+\cdots$ and $f_\f(x)=\f(x)$.
%% PARI function: nph(f, m, n) = subst(f,x,serreverse(f+O(x^n))^m)
\end{remark}

We briefly indicate the contents of this paper. In
Section~\ref{section:earlierwork} we review some of the earlier work
that has been done on B\"ottcher coordinates in the $p$-adic and
characteristic~$p$ setting, after which Section~\ref{section:inverses}
contains some useful facts concerning inverses of various types of
power series.  This is followed in
Sections~\ref{section:pforthmm},~\ref{section:pforthmp},
and~\ref{theorem:radiusp2} with the proofs, respectively, of
Theorems~\ref{theorem:fphiinRp},~\ref{theorem:compm}
and~\ref{theorem:radiusp2}.  In Section~\ref{section:botradconv} we use our
earlier results to prove Corollary~\ref{corollary:botradconv}.
Finally, in Section~\ref{section:xp2conj} we give various precise
conjectures describing the coefficients of the B\"ottcher coordinate
for maps of the form $x^{p^2}+p^{r+2}x^{p^2+1}$.

%%%%%%%%%%%%%%%%%%%%%%%%%%%%%%%%%%%%%%%%%%%%%%%%%%%%%%%%%%%%%%%%%%%%%%
\section{Earlier and Related Work}
\label{section:earlierwork}
%%%%%%%%%%%%%%%%%%%%%%%%%%%%%%%%%%%%%%%%%%%%%%%%%%%%%%%%%%%%%%%%%%%%%%

%% *** MIKE ZIEVE suggested writing everything as power series in $\binom{x}{k}$ and
%% use differencing operator directly, maybe the proof gets easier/neater.
%% I tried doing this, but it doesn't seem to help.

In this section we briefly summarize earlier work on $p$-adic and
characteristic~$p$ B\"ottcher coordinates and relate it to the present
paper.  B\"ottcher coordinates of polynomials over $p$-adic fields
appear to have first been studied by Ingram~\cite{MR3064415} in the
case that the ramification degree~$m$ is relatively prime to~$p$. This
 work was extended and generalized by DeMarco, Ghioca, Krieger,
Nguyen, Tucker, and Ye~\cite{arxiv1703.05365} in two ways. First, they
allow~$m$ to be divisible by~$p$, and second, they work uniformly in
families of polynomials.  Both of these earlier papers consider only
the B\"ottcher coordinate of a monic polynomial in a neighborhood
of~$\infty$, i.e., they restrict attention to rational functions
having a \emph{totally ramified} fixed point. This contrasts with our
results, which apply in particular to rational functions having a
critical fixed point that need not be totally ramified.

We state the result of DeMarco et al., which generalizes
Ingram~\cite[Theorem 2]{MR3064415}, but we conjugate by $x\to x^{-1}$
so as to move their (totally) ramified fixed point to~$0$.

\begin{theorem}[DeMarco et al.\ {\cite[Theorem 6.5]{arxiv1703.05365}}]
\label{theorem:demarcoetal}
Let~$m\ge2$, let $\b_1,\ldots,\b_m\in\CC_p$, let~$\|\,\cdot\,\|_p$ be
the usual absolute value on~$\CC_p$ normalized so
that~$\|p\|_p=p^{-1}$, and let $\f(x)\in\CC_p[\![x]\!]$ be the Taylor series
around~$0$ of the rational function
\[
  \frac{x^m}{1+\b_1x+\b_2x^2+\cdots+\b_mx^m} \in \CC_p[x].
\]
Set
\[
\|\bfbeta\|_p := \max\bigl\{1,\|\b_1\|_p,\ldots,\|\b_m\|_p\bigr\},
\]
and let~$\Dcal_\f$ be the disk
\[
\Dcal_\f :=
\begin{cases}
  \bigl\{ x\in\CC_p : \|x\|_p < \|\bfbeta\|_p^{-1} \bigr\}
  &\text{if $p\nmid m$,} \\[1\jot]
  \bigl\{ x\in\CC_p : \|x\|_p < p^{-1/(p-1)}\|m\|_p\|\bfbeta\|_p^{-1} \bigr\}  
  &\text{if $p\mid m$.} \\
\end{cases}
\]
Then the B\"ottcher coordinate~$f_\f(x)$ converges on~$\Dcal$ and defines
an injective map $f_\f:\Dcal\hookrightarrow\CC_p$.
\end{theorem}

We observe that if we apply Theorem~\ref{theorem:demarcoetal} in the
case that~$\b_1,\ldots,\b_m$ are~$p$-integral, then~$\f(x)$
has~$p$-integral coefficients, and the convergence results in
Theorem~\ref{theorem:demarcoetal}  are the same as those obtained in
Corollary~\ref{corollary:botradconv}(a).

However, even if we assume that~$\b_1,\ldots,\b_m$ are highly
$p$-divisible in Theorem~\ref{theorem:demarcoetal}, we obtain no
improvement in the disk~$\Dcal$, since we always
have~$\|\bfbeta\|_p\ge1$. In particular,
Theorem~\ref{theorem:demarcoetal} does not imply the stronger
convergence estimates given in
Corollary~\ref{corollary:botradconv}(b,c).  On the other hand, the
results of Ingram and of DeMarco et al.\ do apply to series whose
coefficients are not necessarily $p$-integral, a situation that we do
not consider in the present paper. Roughly speaking, the earlier
papers include (polynomial) maps having bad reduction, while our
results deal with maps having good reduction, and we give improved
estimates in the case that $\f(x)\equiv x^m\pmodintext{m}$, which one might
say is the case that~$\f(x)$ has ``very good reduction.''

We note again that the work of Ingram~\cite{MR3064415} and DeMarco et
al~\cite{arxiv1703.05365} deal only with the B\"ottcher coordinates of
polynomial functions, i.e., rational functions having a \emph{totally
  ramified} fixed point. This contrasts with our results, which apply
in particular to rational functions having a critical fixed point that
need not be totally ramified, although we expect that their arguments
could be  adapted to the more general setting.

We also note that these earlier papers construct the B\"ottcher
coordinate via the classical limit~\eqref{eqn:bottaslim} mentioned in
the introduction. This may well have some technical advantages, but it
seems that in order to study subtler $p$-adic properties of B\"ottcher
coordinates, it is necessary to use the finer combinatorial
information provided by the recursive formula for the coefficients of
the B\"ottcher coordinate, as is done in the present paper.

\begin{remark}
Superattracting germs in characteristic~$p$ present many additional
complications if the ramification index is divisible by~$p$.  In
particular, the B\"ottcher coordinate need not exist, and one obtains
interesting parameter and moduli spaces of B\"ottcher-like
coordinates. The case~$\f(x)\in x^p+x^{p+1}K[\![x]\!]$ was studied by
Spencer in his thesis~\cite{spencer:thesis}, and the general case was
investigated by Ruggiero in~\cite{MR3394117}. 
\end{remark}

\begin{remark}
B{\"o}ttcher coordinates have also been studied for higher dimensional
maps~$\CC^N\to\CC^N$; see for example~\cite{MR3119600}. It would be
interesting to investigate the higher dimensional situation for
non-archimedean fields, but we do not do so in the present paper.
\end{remark}

\begin{remark}
There is, of course, a wide body of work on linearization of maps at
\emph{non-critical} fixed points, i.e., at fixed points that are not
superattracting, in both the complex and the non-archimedean
settings. For the latter, see for
example~\cite{arxiv:0808.3266,arxiv0805.1560}.
\end{remark}

%%%%%%%%%%%%%%%%%%%%%%%%%%%%%%%%%%%%%%%%%%%%%%%%%%%%%%%%%%%%%%%%%%%%%%
\section{Inverses of power series}
\label{section:inverses}
%%%%%%%%%%%%%%%%%%%%%%%%%%%%%%%%%%%%%%%%%%%%%%%%%%%%%%%%%%%%%%%%%%%%%%
In this section we describe various sets of power series that are
invariant under taking inverses.  We will use the following well-known
formula for the $k$'th derivative of the composition of functions.

\begin{lemma}[Formula of Fa\'a di Bruno and Arbogast]
\label{lemma:dibruno}
Let $F(x)$ and $G(x)$ be functions that are sufficiently differentiable.
Then the $k$'th derivative of the composition~$F\circ G$ is given by the formula
\begin{multline}
  \label{eqn:faabruno}
  D_x^k F(G(x))
  =\sum_{1\cdot e_1+2\cdot e_2+3\cdot e_3+\cdots+k\cdot e_k=k}  \frac{k!}{e_1!\,1!^{e_1}\,e_2!\,2!^{e_2}\,\cdots\,e_k!\,k!^{e_k}}\\*
  \cdot
  (D_x^{e_1+\cdots+e_k}F)(G(x))\cdot
  \prod_{j=1}^k\left(D_x^jG(x)\right)^{e_j}.
\end{multline}
\end{lemma}
\begin{proof}
See~\cite{MR1903577,MR602839}, for example, for proofs of this formula, 
which dates to the 19'th century.
\end{proof}

Parts~(a) and~(b) of the next proposition are well-known results,
but~(c) and~(d), which are essentially equivalent one another, are
less so.  In particular, we do not see an easy way to use~(b), or its
usual proof by induction as in~\cite[Lemma~IV.5.4]{MR2514094}, to
prove~(c).

\begin{proposition}
\label{proposition:finverseform}
Let~$R$ be a ring of characteristic~$0$. Then each of the following sets
of power series~$\Pcal_1(R),\ldots,\Pcal_4(R)$ in~$R[\![x]\!]$ satisfies
\[
  f(x) \in \Pcal_i(R) \quad\Longleftrightarrow\quad f^{-1}(x)\in \Pcal_i(R).
\]
\vspace{-10pt}
\begin{parts}
\Part{(a)}
$\displaystyle\Pcal_1(R) := \left\{ x\sum_{k=0}^\infty a_kx^k : \text{$a_0=1$ and $a_1,a_2,\ldots\in R$} \right\}$. 
\Part{(b)}
$\displaystyle\Pcal_2(R) := \left\{ x\sum_{k=0}^\infty \frac{a_kx^k}{(k+1)!} : \text{$a_0=1$ and $a_1,a_2,\ldots\in R$} \right\}$. 
\Part{(c)}
$\displaystyle\Pcal_3(R) := \left\{ x\sum_{k=0}^\infty \frac{a_kx^k}{k!} : \text{$a_0=1$ and $a_1,a_2,\ldots\in R$} \right\}$. 
\Part{(d)}
$\displaystyle\Pcal_4(R) := \left\{ x\sum_{k=0}^\infty \frac{a_kx^k}{m^kk!} : \text{$a_0=1$ and $a_1,a_2,\ldots\in R$} \right\}$. 
\end{parts}
\end{proposition}
\begin{proof}
(a)\enspace
After reindexing, this is \cite[Lemma~IV.2.4]{MR2514094}.
\par\noindent(b)\enspace
Again after reindexing, this is \cite[Lemma~IV.5.4]{MR2514094}.  
\par\noindent(c)\enspace  
As noted earlier, we do not see a way to use~(b) to prove~(c), so we give
a direct proof.  We let~$K$ be the fraction field of~$R$, and we set the notation
\[
  f(x)=x\sum_{k=0}^\infty\frac{a_kx^k}{k!}\in\Pcal_3(R)
  \;\text{and}\;
  g(x):=f^{-1}(x)=x\sum_{k=0}^\infty\frac{b_kx^k}{k!}\in K[\![x]\!].
\]
The facts that $f\bigl(g(x)\bigr)=x$ and $a_0=1$ imply that $b_0=1$, so
it remains to prove that every~$b_k$ is in~$R$.

We set some useful notation. For any list
$\bfe=(e_1,e_2,\ldots)$ of non-negative integers, with only finitely
many non-zero entries, we let
\begin{align*}
  \s(\bfe) &= e_1 + e_2 + e_3 + \cdots\,,\\
  \nu(\bfe) &= 1\cdot e_1 + 2\cdot e_2 + 3\cdot e_3 + \cdots\,.
\end{align*}
We apply the formula of Fa\'a di Bruno (Lemma~\ref{lemma:dibruno}) to compute
the $n$'th derivative of the composition~$f\bigl(g(x)\bigr)$ and
evaluate at~$0$, where we note that with our labeling, we have
\[
  D_x^kf(0) = ka_{k-1}
  \quad\text{and}\quad
  D_x^kg(0) = kb_{k-1}.
\]
This yields
\begin{align*}
  D_x^n(f\circ g)(0)
  &= \sum_{\nu(\bfe)=n}   n! \prod_{j=1}^n \frac{1}{j!^{e_j}\cdot e_j!} \cdot
    (D_x^{\s(\bfe)}f)(g(0))\cdot
    \prod_{j=1}^n\left(D_x^jg(0)\right)^{e_j} \\
  &=\sum_{\nu(\bfe)=n}
  n! \prod_{j=1}^n \frac{1}{j!^{e_j}\cdot e_j!} \cdot \s(\bfe)
  a_{\s(\bfe)-1} \prod_{j=1}^n (jb_{j-1})^{e_j}.
\end{align*}
On the other hand, since~$f\bigl(g(x)\bigr)=x$, we have $D_x^n(f\circ
g)=0$ for all $n\ge2$. This yields a recursion for the coefficients
of~$g(x)$.

Thus the term with $\nu(\bfe)=n$ and $\s(\bfe)=1$, i.e., the term with
$e_n=1$ and all other $e_j=0$, is
\[
  n!\cdot \frac{1} {n!^1\cdot 1!} \cdot 1 \cdot a_0 \cdot nb_{n-1} = nb_{n-1},
\]
so if we assume (by induction) that $b_1,\ldots,b_{n-2}\in R$,
and use $a_0=b_0=1$ and $a_k\in R$ for all~$k$, we see that in order
to prove that $b_{n-1}\in R$, it suffices to show that for every~$\bfe$
satisfying $\nu(\bfe)=n$, we have
\[
  n! \prod_{j=1}^n \frac{1}{j!^{e_j}\cdot e_j!} \cdot
  \s(\bfe) \cdot \prod_{j=1}^n j^{e_j}
  \equiv 0 \pmod{n}. 
\]
The validity of this congruence, which is by no means obvious, is proven
in a separate lemma at the end of this section; see Lemma~\ref{lemma:combin}.
\par\noindent(d)\enspace  
This follows from~(c). Thus let $f(x)=x\sum a_kx^k/m^kk!\in\Pcal_4(R)$.
Then $F(x):=x\sum a_kx^k/k!\in\Pcal_3(R)$, so~(c) tells us
hat~$F^{-1}(x)=x\sum b_kx^k/k!$ with~$b_0=1$ and all~$b_k\in R$.
Using the identity $f(x)=mF(x/m)$, we see that
\[
  f^{-1}(x) = mF^{-1}(x/m) = m\cdot\frac{x}{m} \sum_{k=0}^\infty \frac{b_k(x/m)^k}{k!}
  = x \sum_{k=0}^\infty \frac{b_kx^k}{m^kk!},
\]
which proves that~$f^{-1}(x)\in\Pcal_4(R)$.
\end{proof}

\begin{lemma}
\label{lemma:combin}
Let $e_1,e_2,\ldots,e_n$ be non-negative integers satisfying
\[
  1\cdot e_1 + 2\cdot e_2 + \cdots + n\cdot e_n = n.
\]
Then
\[
  n! \prod_{j=1}^n \frac{1}{j!^{e_j}\cdot e_j!} \cdot
  (e_1+e_2+\cdots+e_n) \cdot \prod_{j=1}^n j^{e_j}
  \equiv 0 \pmod{n}. 
\]
\end{lemma}
\begin{proof}
To ease notation, we let
\[
  m = e_1+e_2+\cdots+e_n.
\]
For the given list $\bfe=(e_1,\ldots,e_n)$,
let~$Z(\bfe)$ be the collection of all expressions of the form
\[
\Bigl( (S_1,s_1), \Bigl\{ (S_2,s_2), \ldots, (S_m,s_m) \Bigr\} \Bigr),
\]
where:
\begin{parts}
\Part{(i)}
$S_1,\ldots,S_m$ are disjoint subsets of $\ZZ/n\ZZ$.
\Part{(ii)}
For each $1\le i\le m$ we have $s_i\in S_i$.
\Part{(iii)}
For each $1\le k\le n$ we have $e_k = \#\{1\le i\le m : \#S_i=k\}$.
\end{parts}
We observe that~(i) and~(iii) imply that~$\ZZ/n\ZZ$ is equal to the
disjoint union of~$S_1,\ldots,S_m$. We also note that an element
of~$Z(\bfe)$ consists of a distinguished pointed set~$(S_1,s_1)$ and an
unordered collection of~$m-1$ additional pointed
sets $(S_2,s_2),\ldots,(S_m,s_m)$.

We claim that
\begin{equation}
  \label{eqn:numZe}
  \#Z(\bfe) =   n! \prod_{j=1}^n \frac{1}{j!^{e_j}\cdot e_j!} \cdot
  m \cdot \prod_{j=1}^n j^{e_j},
\end{equation}
i.e., we claim that~$\#Z(\bfe)$ is exactly the quantity that we are
trying to prove is divisible by~$n$. To see this, we count the number
of ways to create an element of~$Z(\bfe)$. First we take~$\ZZ/n\ZZ$
and partition it into~$m$ subsets $S_1,\ldots,S_m$ consisting of~$e_1$
subsets containing~$1$ element,~$e_2$ subsets containing~$2$ elements,
etc. The number of ways to do this if we keep track of the order of
$S_1,\ldots,S_m$ is the multinomial coefficient
\[
  \binom{n}{1,1,\ldots,1,2,2,\ldots,2,\ldots} = \frac{n!}{1!^{e_1}\cdot 2!^{e_2} \cdots n!^{e_n}}.
\]
However, if we don't care about the order, then we need to divide
by~$e_1!$ to account for reordering the 1-element subsets and
by~$e_2!$ to account for reordering the 2-element subsets, etc. This
accounts for the factor of $\prod e_j!$ in the denominator of~\eqref{eqn:numZe}. Next, we
actually want $S_1,\ldots,S_m$ to be pointed subsets, so we choose one
element from each~$S_i$. This can be done in
$\prod\#S_i=\prod{j}^{e_j}$ ways, where the equality follow
from~(iii). Finally we need to choose one of the pointed
sets~$(S_i,s_i)$ to be distinguished, which can be done in~$m$
ways. This concludes the verification of~\eqref{eqn:numZe}.

For any subset~$S\subset\ZZ/n\ZZ$, we let
\[
  S+1 := \{ s+1\bmod n : s\in S \}.
\]
We define an ``increment-by-$1$'' operation on~$Z(\bfe)$ by
the formula
\begin{multline*}
  \hfill I : Z(\bfe) \xrightarrow{\;\;\sim\;\;} Z(\bfe), \hfill\\
  I\Bigl((S_1,s_1), \Bigl\{(S_2,s_2),\cdots,(S_m,s_m)\Bigr\} \Bigr) \hfill\\
  \hfill{} =   \Bigl((S_1+1,s_1+1), \Bigl\{(S_2+1,s_2+1),\cdots,(S_m+1,s_m+1)\Bigr\} \Bigr).
\end{multline*}
Now the key observation, for which we thank Melody Chan, is that the
increase-by-1 operator is a permutation of the set $Z(\bfe)$ for which
the orbit of every element has size exactly~$n$. This will obviously
imply that~$n$ divides~$\#Z(\bfe)$, so it remains to prove that truth
of this observation.

Since we are working in~$\ZZ/n\ZZ$, it is clear that the~$n$'th
iterate~$I^n$ acts as the identity map on~$Z(\bfe)$, so it suffices to
prove that if~$I^k$ fixes an element
\[
  \Bigl((S_1,s_1), \Bigl\{(S_2,s_2),\cdots,(S_m,s_m)\Bigr\} \Bigr) \in Z(\bfe),
\]
then~$n\mid k$. But if~$I^k$ fixes this element, then
incrementing by~$k$ must fix the distinguished point~$s_1$ of the
distinguished pointed subset~$(S_1,s_1)$, i.e., 
\[
  s_1\equiv s_1+k \pmodintext{n}.
\]
Hence~$n\mid k$, which completes the proof of Lemma~\ref{lemma:combin}.
\end{proof}

\begin{question}
Continuing with the notation in Lemma~\ref{lemma:combin},
if $e_1<n$, is it true that we always have
\[
  n! \prod_{j=1}^n \frac{1}{j!^{e_j}\cdot e_j!} \cdot \prod_{j=1}^n j^{e_j}
  \equiv 0 \pmod{n} ?
\]
We have verified this stronger result experimentally for various
values of~$n$ and~$\bfe$.
\end{question}

%%%%%%%%%%%%%%%%%%%%%%%%%%%%%%%%%%%%%%%%%%%%%%%%%%%%%%%%%%%%%%%%%%%%%%
\section{Proof of Theorem \ref{theorem:fphiinRp}}
\label{section:pforthmm}
%%%%%%%%%%%%%%%%%%%%%%%%%%%%%%%%%%%%%%%%%%%%%%%%%%%%%%%%%%%%%%%%%%%%%%

We start by setting some notation for truncating and for picking out
coefficients of power series. This notation will be used in this
section and in Section~\ref{section:nonint}.  For any power
series~$P(x)=\sum_{k\ge0} c_kx^k$, we write
\begin{equation}
  \label{eqn:xkcoeff}
  P^{[k]}(x) = \sum_{i=0}^k c_ix^i \quad\text{and}\quad   P(x)[x^k] = c_k.
\end{equation}

\begin{proof}[Proof of Theorem~$\ref{theorem:fphiinRp}$]
Write
\begin{equation}
  \label{eqn:phixp1ppsi}
  \f(x) = x^p \bigl( 1 + p\psi(x) \bigr)\quad\text{with $\psi(x)\in xR[\![x]\!]$.}
\end{equation}
We assume that $f_\f^{[k]}(x)\in R[x]$, and we proceed to prove
that $f_\f^{[k+1]}(x)\in R[x]$.  To ease notation, we write
\[
  f(x) = f_\f^{[k]}(x) + \b x^{k+1} + \cdots\quad\text{with $\b=a_{k+1}$.}
\]
Then~$\b$ is determined by its appearance in the $x^{p+k}$ coefficient
of the defining relation~\eqref{eqn:bottdef}. Thus
\begin{align*}
  0 &= \Bigl( \f\bigl(f_\f^{[k]}(x)+\b x^{k+1}\bigr) - f_\f^{[k]}(x^p) \Bigr) [x^{p+k}]
     \quad\text{using \eqref{eqn:bottdef},} \\*
  &= \Bigl( \bigl(f_\f^{[k]}(x)+\b x^{k+1}\bigr)^p \bigl( 1 + p \psi\bigl(f_\f^{[k]}(x)+\b x^{k+1}\bigr) \bigr) 
     - f_\f^{[k]}(x^p) \Bigr) [x^{p+k}] \\*
  &\omit\hfil\text{using \eqref{eqn:phixp1ppsi},} \\
  &= \Bigl( \bigl( f_\f^{[k]}(x)^p + p\b f_\f^{[k]}(x)^{p-1} x^{k+1} \bigr) \bigl( 1 + p \psi\bigl(f_\f^{[k]}(x)+\b x^{k+1}\bigr) \bigr) \\*
     &\omit\hfil${}- f_\f^{[k]}(x^p) \Bigr) [x^{p+k}]$\\*
     &\omit\hfill since $f_\f^{[k]}(x)\in xR[x]$, \\
  &= \Bigl( \bigl( f_\f^{[k]}(x)^p + p\b x^{p+k} \bigr) \bigl( 1 + p \psi\bigl(f_\f^{[k]}(x)+\b x^{k+1}\bigr) \bigr) 
         - f_\f^{[k]}(x^p) \Bigr) [x^{p+k}] \\*
     &\omit\hfill since $f_\f^{[k]}(x)\in x+x^2R[x]$, \\
  &= \Bigl( \bigl( f_\f^{[k]}(x)^p + p\b x^{p+k} \bigr) \bigl( 1 + p \psi\bigl(f_\f^{[k]}(x)\bigr) \bigr) 
         - f_\f^{[k]}(x^p) \Bigr) [x^{p+k}] \\*
     &\omit\hfill since $\bigl( f_\f^{[k]}(x)^p + p\b x^{p+k} \bigr)\in x^pR[x]$ and $\psi(x)\in xR[x]$,\\
  &= \Bigl(  f_\f^{[k]}(x)^p + p\b x^{p+k}  +  f_\f^{[k]}(x)^p  p \psi\bigl(f_\f^{[k]}(x)\bigr)
         - f_\f^{[k]}(x^p) \Bigr) [x^{p+k}]\\*
     &\omit\hfill since  $\psi\bigl(f_\f^{[k]}(x)\bigr)\in xR[x]$.
\end{align*}
Hence
\[
  \b = a_{k+1}
  = \frac{1}{p}\Bigl(f_\f^{[k]}(x^p)-f_\f^{[k]}(x)^p\Bigr)[x^{p+k}]
  - \Bigl(f_\f^{[k]}(x)^p \psi\bigl(f_\f^{[k]}(x)\bigr)\Bigr)[x^{p+k}].
\]
The assumptions that~$f_\f^{[k]}$ and $\psi(x)$ have coefficients
in~$R$ imply that the second term has coefficients in~$R$. As for the
first term, it has coefficients in~$R$, since for any
polynomial~$F(x)\in R[x]$, we have
\begin{equation}
  \label{eqn:abpapbp}
  F(x)^p \equiv F(x^p) \pmod{pR[x]}.
\end{equation}
We note that~\eqref{eqn:abpapbp} is the property that requires the
assumption that $a^p\equiv a\pmodintext{pR}$ for every~$a\in R$.
Without that assumption, we would only have $F(x)^p \equiv \tilde
F(x^p) \pmodintext{pR[x]}$, where~$\tilde F$ is obtained from~$F$ by
raising its coefficients to the $p$th power.
This completes the proof of Theorem~\ref{theorem:fphiinRp}. 
\end{proof}

%%%%%%%%%%%%%%%%%%%%%%%%%%%%%%%%%%%%%%%%%%%%%%%%%%%%%%%%%%%%%%%%%%%%%%
\section{Proof of Theorem \ref{theorem:compm}}
\label{section:pforthmp}
%%%%%%%%%%%%%%%%%%%%%%%%%%%%%%%%%%%%%%%%%%%%%%%%%%%%%%%%%%%%%%%%%%%%%%

We start with an elementary, but useful, description of the $k$'th
derivative of the $m$'th power of a function.  For notational
convenience, we let $D_x$ denote differentiation with respect
to~$x$. In particular, we note that~$D_x$ operates formally on power
series rings such as~$K[\![x]\!]$.

\begin{definition}
For integers $m\ge1$ and $k\ge1$, define a submodule of the graded
polynomial ring~$\ZZ[T_0,T_1,T_2,\ldots,T_r]$ by
\begin{multline*}
  \ZZ[T_0,T_1,T_2,\ldots,T_r]^{[m,k]} \\*
  := \operatorname{Span}_\ZZ
  \left\{ T_0^{e_0}T_1^{e_1}T_2^{e_2}\cdots T_r^{e_r} :
  \sum_{\ell=0}^r e_\ell = m \text{ and } \sum_{\ell=0}^r\ell e_\ell = k \right\}.
\end{multline*}
In other words, $\ZZ[T_0,\ldots,T_r]^{[m,k]}$ is the span of the
monomials of degree~$m$ and weight~$k$, according to the grading
$\wt(T_\ell)=\ell$.
\end{definition}

%% \begin{definition}
%% Let $R$ be a ring. The \emph{weight} of a monomial in $R[Y_0,Y_1,\ldots]$
%% is defined to be
%% \[
%%   \wt\left(Y_0^{e_0}Y_1^{e_1}Y_2^{e_2}\cdots Y_r^{e_r}\right)
%%   := e_1 + 2e_2 + \cdots + re_r.
%% \]
%% We say that $f\in R[Y_0,Y_1,\ldots]$ is homogeneous of weight~$k$ if
%% it is a sum of weight~$k$ monomials.
%% \end{definition}

\begin{lemma}
\label{lemma:Dxkymstrong}
Let $y$ be a sufficiently differentiable function of $x$, let $k\ge1$,
and let $m\ge1$.  There is a polynomial
$\D_{m,k}\in\ZZ[T_0,\ldots,T_{k-1}]^{[m,k]}$ such that
\[
  D_x^k(y^m) = my^{m-1}D_x^k(y) +  m\D_{m,k}(y,D_x^{\vphantom1}y,D_x^2y,\ldots,D_x^{k-1}y).
\]
\end{lemma}
\begin{proof}
We induct on~$k$. For~$k=1$ this is just the chain rule
$D_x(y^m)=my^{m-1}D_x(y)$, so $\D_{m,1}=0$.  Assume true for~$k$.
We observe that
\[
  D_x\bigl(my^{m-1}D_x^k(y)\bigr)
  = my^{m-1}D_x^{k+1}(0) + m(m-1)y^{m-2}D_x(y)D_x^k(y),
\]
while differentiating a monomial of degree~$m$ and weight~$k$ yields a
sum of monomials that have degree~$m$ and weight~$k+1$.  Hence
$D_x^{k+1}(y^m)$ has the desired form.
\end{proof}

\begin{remark}
We note that Lemma~\ref{lemma:Dxkymstrong} may also be derived as a
special case of the classical formula of Fa\'a di Bruno
for the $k$'the derivative of a
composition of functions.
Thus setting 
$F(x)=x^m$, with~$G(x)=y$
in Lemma~\ref{lemma:dibruno},
we see that the term in~\eqref{eqn:faabruno}
with~$e_k=1$, and necessarily $e_1=\cdots=e_{k-1}=0$, is
\[
  \frac{k!}{e_k!(k!)^{e_k}} (D_xF)(y)\cdot (D_x^k(y))^{e_k}
  = my^{m-1} D_x^k(y),
\]
matching the initial term in Lemma~\ref{lemma:Dxkymstrong}. And the
other terms have the desired form once we note that for any
$\ell\ge1$, the term
\[
  (D_x^\ell F)(y) = m(m-1)\cdots(m-\ell+1)y^{m-\ell}
\]
has the desired factor of~$m$, and that the quantity $(D^j_xy)/j!$ is
integral.
\end{remark}  

We are now ready to prove Theorem~$\ref{theorem:compm}$.

\begin{proof}[Proof of Theorem~$\ref{theorem:compm}$]
It turns out to be
somewhat easier to work with the inverse B\"ottcher coordinate, so we
let $F(x) = f_\f^{-1}(x)\in K[\![x]\!]$. Thus~$F$ is determined by the
functional equation
\begin{equation}
  \label{eqn:FphiFm}
  F\bigl(\f(x)\bigr) = F(x)^m.
\end{equation}
We define series~$\nu(x)$ and~$G(x)$ by
\[
  \f(x) = x^m\nu(x) = x^m\sum_{i=0}^\infty b_i x^i
  \quad\text{and}\quad
  F(x) = xG(x) = x\sum_{\ell=0}^\infty \frac{c_\ell}{\ell!}x^\ell.
\]
Substituting these expressions for~$F$ and~$\f$
into~\eqref{eqn:FphiFm} and canceling~$x^m$ yields
\begin{equation}
  \label{eqn:nuGxmnuGm}
  \nu(x) G\bigl(x^m\nu(x)\bigr) = G(x)^m.
\end{equation}
This functional equation determines the coefficients of~$G(x)$ in terms of
the coefficients of~$\nu(x)$.

We are going to take the $k$'th derivative of both sides
of~\eqref{eqn:nuGxmnuGm} and evaluate at~$x=0$.  For the right-hand
side we use Lemma~\ref{lemma:Dxkymstrong} with $y=G(x)$. Noting that
$(D_x^\ell G)(0)=c_\ell$ and $c_0=1$,  we obtain
\begin{equation}
  \label{equation:DkxGxmx0}
  D_x^k\bigl(G(x)^m\bigr) \Bigr|_{x=0} 
  = mc_k + m \D_{m,k}(c_0,c_1,\ldots,c_{k-1}),
\end{equation}
where $\D_{m,k}(T_0,\ldots,T_{k-1})\in\ZZ[T_0,\ldots,T_{k-1}]^{[m,k]}$
is as defined in Lemma~\ref{lemma:Dxkymstrong}.

In order to handle the left-hand side of~\eqref{eqn:nuGxmnuGm}, we
expand~$G\bigl(x^m\nu(x)\bigr)$ as a series and differentiate. Thus
\begin{align*}
  D_x^k\Bigl(\nu(x)G\bigl(x^m\nu(x)\bigr)\Bigr)
  &= D_x^k \left( \nu(x) \sum_{\ell=0}^\infty \frac{c_\ell}{\ell!}\bigl(x^m\nu(x)\bigr)^\ell \right) \\*
  &= \sum_{\ell=0}^\infty \frac{c_\ell}{\ell!} D_x^k\Bigl(x^{m\ell} \nu(x)^{\ell+1}\Bigr) \\*
  &= \sum_{\ell=0}^\infty \frac{c_\ell}{\ell!} \sum_{i+j=k} \binom{k}{i} D_x^i(x^{m\ell})\cdot D_x^j\bigl(\nu(x)^{\ell+1}\bigr).
\end{align*}

Evaluating at $x=0$, we note that
\[
  D_x^i(x^{m\ell})\bigr|_{x=0} = \begin{cases} (m\ell)!&\text{if $i=m\ell$,}\\ 0&\text{otherwise,}\\ \end{cases}
\]
so the only term that remains in the inner sum is
$(i,j)=(m\ell,k-m\ell)$. Further, since $0\le j\le k$, this term
appears only if $\ell\le k/m$. Hence
\[
  \left.   D_x^k\Bigl(\nu(x)G\bigl(x^m\nu(x)\bigr)\Bigr) \right|_{x=0}
  = \sum_{\ell=0}^{\lfloor k/m\rfloor} \frac{c_\ell}{\ell!}\cdot\binom{k}{m\ell}
  (m\ell)!\cdot  D_x^{k-m\ell}\bigl(\nu(x)^{\ell+1}\bigr) \Bigr|_{x=0}.
\]

We next observe that
\[
  \frac{(m\ell)!}{m!\ell!} \in \ZZ\quad\text{for $\ell\ge1$,}
\]
so pulling off the $\ell=0$ term and using the fact that~$c_0=1$, we find that
\begin{equation}
  \label{eqn:DxknuGxmx0}
  \left.   D_x^k\Bigl(\nu(x)G\bigl(x^m\nu(x)\bigr)\Bigr) \right|_{x=0}
  \in k!b_k +  \sum_{\ell=1}^{\lfloor k/m\rfloor} m!\cdot c_\ell \cdot \ZZ[b_1,b_2,\ldots].
\end{equation}

Subsituting~\eqref{equation:DkxGxmx0} and \eqref{eqn:DxknuGxmx0} into
the $k$'th derivative of~\eqref{eqn:nuGxmnuGm} evaluated at~$x=0$, and
dividing by~$m$, we obtain 
\begin{equation}
  \label{eqn:ckDmk}
  c_k \in  -\D_{m,k}(c_0,c_1,\ldots,c_{k-1}) + 
  \frac{k!b_k}{m} + \sum_{\ell=1}^{\lfloor k/m\rfloor}
  c_\ell \cdot \ZZ[b_1,b_2,\ldots] .
\end{equation}

We now proceed by induction, starting from~$c_0=1$. If we assume that
$k!b_k\in mR$, as in part~(b), then it is clear from~\eqref{eqn:ckDmk}
that
\[
  c_0,c_1,\ldots, c_{k-1}\in R \Longrightarrow c_k\in R.
\]
Hence under the assumption in~(b), we see
that~$f_\f^{-1}(x)=\sum c_kx^k/k!$ with $c_k\in R$.

In order to prove~(a), our induction hypothesis is that $c_\ell\in
m^{-\ell}R$ for all $0\le\ell<k$, and our goal is to conclude that
$c_k\in m^{-k}R$. We consider the integrality properties of each of
the terms in~\eqref{eqn:ckDmk}. First, the
term~$\D_{m,k}(c_0,\ldots,c_{k-1})$ is a $\ZZ$-linear combination of
monomials of the form
\[
  c_0^{e_0}c_1^{e_1}\cdots c_{k-1}^{e_{k-1}}\quad\text{with}\quad \sum_{\ell=0}^{k-1} \ell e_\ell = k.
\]
Writing $c_\ell=\g_\ell/m^\ell$ with $\g_\ell\in R$, we see that
\[
  c_0^{e_0}c_1^{e_1}\cdots c_{k-1}^{e_{k-1}}
  = \frac{\g_0^{e_0}\g_1^{e_1}\cdots \g_{k-1}^{e_{k-1}}}{ m^{0\cdot e_0+1\cdot e_1+2\cdot e_2+\cdots+{(k-1)}e_{k-1}} }
  = \frac{\g}{m^k}\quad\text{with $\g\in R$.}
\]
Hence $\D_{m,k}(c_0,\ldots,c_{k-1})\in m^{-k}R$.

Next, since $k\ge1$ and $b_k\in R$, we see that $k!b_k/m\in
m^{-k}R$. Finally, we note that $c_\ell\cdot\ZZ[b_1,b_2,\ldots]
\subset m^{-\ell}R$ with $\ell\le k/m<k$, so these terms are also
in~$m^{-k}R$. This concludes the proof by induction that $c_k\in m^{-k}R$ for all $k\ge1$.

This completes the proof that the inverse B\"ottcher
coordinate~$f_\f^{-1}(x)$ has the desired form.  To complete the proof
of Theorem~\ref{theorem:compm}, we need merely note that
Proposition~\ref{proposition:finverseform}(c,d) then tells us
that~$f_\f(x)$ also has the desired form.
\end{proof}

%% Old version of the differentiation lemma
%% \begin{lemma}
%% \label{lemma:Dxkym}
%% Let $y$ be a function of $x$. Then for $k\ge1$ and $m\ge1$ we have
%% \[
%%   D_x^k(y^m) \in my^{m-1} D_x^k(y) + m\ZZ[y,D_xy,D_x^2y,\ldots,D_x^{k-1}y].
%% \]
%% \end{lemma}
%% \begin{proof}
%% We induct on~$k$. For~$k=1$ this is just the chain rule
%% $D_x(y^m)=my^{m-1}D_x(y)$. Assume true for~$k$. Then
%% \begin{align*}
%%   D_x^{k+1}(y^m)
%%   &= D_x\Bigl( my^{m-1}D_x^k(y) + m P_k(y,D_xy,\ldots,D_x^{k-1}y)\Bigr) \\
%%   &= my^{m-1}D_x^{k+1}(y) + m(m-1)y^{m-2}D_x^k(y) \\
%%   & \omit\hfill$\displaystyle{}+ m D_x\Bigl(P_k(y,D_xy,\ldots,D_x^{k-1}y)\Bigr)$ \\
%%   &\in my^{m-1}D_x^{k+1}(y) +  m\ZZ[y,D_xy,D_x^2y,\ldots,D_x^ky].
%% \end{align*}
%% \end{proof}

%%%%%%%%%%%%%%%%%%%%%%%%%%%%%%%%%%%%%%%%%%%%%%%%%%%%%%%%%%%%%%%%%%%%%%
\section{Proof of Theorem~$\ref{theorem:radiusp2}$}
\label{section:nonint}
%%%%%%%%%%%%%%%%%%%%%%%%%%%%%%%%%%%%%%%%%%%%%%%%%%%%%%%%%%%%%%%%%%%%%%

\begin{proof}[Proof of Theorem~$\ref{theorem:radiusp2}$]
We will make frequent use of Legendre's formula for the valuation of a
factorial. Writing the base-$p$ expansion of a non-negative
integer~$N$ as
\[
  N=\sum_{i\ge0}N_ip^i\quad\text{with}\quad 0\le N_i<p,
  \quad\text{we let}\quad
  S_p(N) := \sum_{i\ge0} N_i.
\]
Then Legendre's formula says that
\begin{equation}
  \label{eqn:ordpkfact}
  \ord_p(N!) = \frac{N-S_p(N)}{p-1}.
\end{equation}
In particular, 
\[
  \text{$N$ is a power of $p$}\quad\Longrightarrow\quad
  \ord_p(N!) = \frac{N-1}{p-1}.
\]

We note that~$a_0=1$, and one easily checks that $a_1=-1$.  We are
going to prove that if~$k$ is a non-trivial power of~$p$, then
\[
  a_k\equiv a_{k/p}\pmodintext{p}.
\]
Combined with the initial value~$a_1=-1$, this clearly implies that
$a_k\equiv -1\pmodintext{p}$ for all~$k$ that are powers of~$p$.

So we take~$k$ to be a power of~$p$ with $k\ge p$. To ease notation
(and with a view to generalizations), we let
\[
  q := p^2.
\]
Expanding the B\"ottcher equation $\f\bigl(f_\f(x)\bigr)=f(x^q)$, we
find that the coefficient of~$x^k$ gives us a recursive formula
for~$a_k$ in terms of the lower~$a_i$. More specifically, we can write~$a_k$
as a sum of three terms
\begin{equation}
  \label{eqn:akAkBkCk}
  a_k = A_k[x^k] - B_k[x^k] - C_k[x^k],
\end{equation}
where as in~\eqref{eqn:xkcoeff}, we write~$P[x^k]$ for the coefficient
of~$x^k$ in a polynomial or power series~$P(x)$, and
where~$A_k$,~$B_k$, and~$C_k$ are given by the formulas
\[
  A_k := \frac{k!}{q}\sum_{\ell=0}^{k-1} \frac{a_\ell}{\ell!}x^{q\ell},\hspace{.5em}
  B_k := \frac{k!}{q}\left(\sum_{\ell=0}^{k-1} \frac{a_\ell}{\ell!}x^{\ell}\right)^q  \!\!,\hspace{.5em}
  C_k := k!\left(x\sum_{\ell=0}^{k-1} \frac{a_\ell}{\ell!}x^{\ell}\right)^{q+1} \!\!.
\]

We start by analyzing $B_k[x^k]$, since this is the term in which we
will find a single monomial in the~$a_i$ whose coefficient is prime
to~$p$.  Expanding the~$q$'th power that defines~$B_k$ yields
\begin{multline}
  \label{eqn:Bkxkformula}
  B_k[x^k] = \frac{k!}{q} \sum_{\substack{e_0+e_1+\cdots+e_{k-1}=q\\
      0\cdot e_0+1\cdot e_1+\cdots+(k-1)e_{k-1}=k\\}}
  \binom{q}{e_0,e_1,\ldots,e_{k-1}}  \\*
  {}\cdot \left(\frac{a_0}{0!}\right)^{e_0}
  \left(\frac{a_1}{1!}\right)^{e_1}
  \left(\frac{a_2}{2!}\right)^{e_2}
  \cdots
  \left(\frac{a_{k-1}}{(k-1)!}\right)^{e_{k-1}}.
\end{multline}
For any given $\bfe:=(e_0,\ldots,e_{k-1})$, the coefficent of
$\bfa^\bfe:=a_0^{e_0}\cdots a_{k-1}^{e_{k-1}}$ in~$B_k[x^k]$ is
\begin{equation}
  \label{eqn:Bkxkae}
  B_k[x^k][\bfa^\bfe]
  :=
  \frac{k!}{q} \cdot \binom{q}{e_0,e_1,\ldots,e_{k-1}} \cdot \prod_{n=0}^{k-1} \frac{1}{n!^{e_n}}.
\end{equation}
We claim that
\begin{align}
  \ord_p B_k[x^k][\bfa^\bfe] &= 0 &&\text{if $\bfa^\bfe=a_0^{q-p}a_{k/p}^p$,}
  \label{eqn:ordpBkeq0} \\*
  \ord_p B_k[x^k][\bfa^\bfe] &> 0 &&\text{otherwise.}
    \label{eqn:ordpBkgt0} 
\end{align}
In~\eqref{eqn:Bkxkae} we expressed $B_k[x^k][\bfa^\bfe]$ as a product
of three terms.  We compute the valuations of these three terms, using
the fact that~$k$ and~$q$ are powers of~$p$, and using Legendre's
formula~\eqref{eqn:ordpkfact} to compute the valuations of various
factorials.  Thus
\begin{align*}
  \ord_p\left(\frac{k!}{q}\right) &= \frac{k-1}{p-1} - \ord_p(q), \\
%%%%%%
  \ord_p \binom{q}{e_0,e_1,\ldots,e_{k-1}}
  &= \ord_p(q!)-\sum_{n=0}^{k-1} \ord_p(e_n!) \\*
  &= \frac{1}{p-1}\left( (q - 1) - \sum_{n=0}^{k-1} \bigl(e_n - S_p(e_n)\bigr) \right) \\
  &= \frac{1}{p-1}\left(  - 1 + \sum_{n=0}^{k-1}  S_p(e_n) \right) \\*
  &\omit\hfill using $\sum e_n = q$, \\
%%%%%%
  \ord_p\left( \prod_{n=0}^{k-1} \frac{1}{n!^{e_n}}\right)
  &= -\frac{1}{p-1} \sum_{n=0}^{k-1} e_n\bigl(n - S_p(n)\bigr)\\*
  &= \frac{1}{p-1} \left( - k +  \sum_{n=0}^{k-1} e_n S_p(n) \right) \\*
  &\omit\hfill using $\sum ne_n=k$.
\end{align*}
Adding these three pieces and using $\ord_p(q)=2$, we find that
\begin{equation}
  \label{eqn:ordpBka}
  (p-1)\ord_p B_k[x^k][\bfa^\bfe]
  =     - 2p  +  \sum_{n=0}^{k-1} \bigl( e_n S_p(n) + S_p(e_n)  \bigr).
\end{equation}
As a warm-up, we use~\eqref{eqn:ordpBka} to
prove~\eqref{eqn:ordpBkeq0}, which is one of our claims. Thus
\begin{align*}
  (p-1)\ord_p B_k[x^k][a_0^{q-p}a_{k/p}^p]
   &=  - 2p  +  S_p(p^2-p) + p S_p(k/p) + S_p(p)\\*
   &= - 2p + (p-1) + p + 1 = 0.
\end{align*}

Before proceeding, we are going to rewrite~\eqref{eqn:ordpBka} to exploit
the fact that~$S_p(n)\ge1$ for all $n\ge1$. So we pull the $n=0$ term out of the
sum, replace $S_p(n)$ with $S_p(n)-1+1$, and use the fact that $\sum e_n=q=p^2$.
This yields
\begin{multline*}
  (p-1)\ord_p B_k[x^k][\bfa^\bfe] \\
  = p^2 - 2p - e_0 + S_p(e_0) +
  \smash[tb]{\sum_{n=1}^{k-1} \bigl( e_n (S_p(n)-1) + S_p(e_n) \bigr).}
\end{multline*}
Hence
\begin{multline*}
  \ord_p B_k[x^k][\bfa^\bfe] \le 0 \\*
  \quad\Longleftrightarrow\quad
  \sum_{n=1}^{k-1} \bigl( e_n (S_p(n)-1) + S_p(e_n) \bigr) \le e_0 - p^2 + 2p - S_p(e_0).
\end{multline*}
In particular, we have $e_0\ge p^2-2p$,
while $\sum e_n=p^2$ combined with
$\sum ne_n=k>0$ tell us that $e_0<p^2$.
Thus $p^2-2p\le e_0 < p^2$. We split this interval into two pieces.
Thus for $0\le j<p$, we have
\[
\begin{array}{|c|c|c|} \hline
  e_0 & S_p(e_0) & e_0 - p^2+2p-S_p(e_0) \\ \hline\hline
  p^2-2p+j & p-2+j & 2-p \\ \hline
  p^2-p+j & p-1+j & 1 \\ \hline
\end{array}
\]
Since $p\ge2$,  we have proven that
\begin{multline*}
  \ord_p B_k[x^k][\bfa^\bfe] \le 0 \\*
  \quad\Longrightarrow\quad
  %%    \sum_{n=1}^{k-1} \bigl( e_n (S_p(n)-1) + S_p(e_n) \bigr) \le 1.
  \sum_{n=1}^{k-1} \bigl( e_n (S_p(n)-1) + S_p(e_n) \bigr) \le
  \begin{cases}
    0 &\text{if $p^2-2p\le e_0<p^2-p$,} \\
    1 &\text{if $p^2-p\le e_0<p^2$.} \\
  \end{cases}
\end{multline*}
We know from $\sum ne_n=k$ that there exists at least one~$m\ge1$ such
that $e_m\ge1$, and for each such~$m$ we have $S_p(e_m)\ge1$.
Thus there is a unique~$m\ge1$ with $e_m\ge1$.
Further,
we observe that if~$e_m$ is not a power of~$p$, then $S_p(e_m)\ge2$,
so we conclude that~$e_m$ is a power of~$p$. Also, since the sum
is strictly positive, we see that  $p^2>e_0\ge p^2-p$.
We now know the following three facts:
\[
  (1)~p^2 = e_0 + e_m.
  \qquad
  (2)~p^2>e_0\ge p^2-p.
  \qquad
  (3)~\text{$e_m$ is a power of $p$.}
\]
Thus~$e_m$ is a power of~$p$ satisfying $p\ge e_m>0$, which proves
that~$e_m=p$. Then
\[
  k = \sum_{n=0}^{k-1} ne_n = me_m
  \quad\Longrightarrow\quad
  m = k/e_m = k/p.
\]
This proves that
\[
  \ord_p B_k[x^k][\bfa^\bfe] \le 0
  \quad\Longrightarrow\quad
  \bfa^\bfe = a_0^{q-p}a_{k/p}^p,
\]
which concludes the proof of~\eqref{eqn:ordpBkgt0}.

We next consider~$A_k$. For $k$ a power of~$p$, we have
\[
A_k[x^k] = \begin{cases}
  0 & \text{if $k<q$,} \\
  \dfrac{k! a_{k/q}}{q(k/q)!} & \text{if $k\ge q$.} \\
\end{cases}
\]
For $k\ge q = p^2$ we compute
\begin{align*}
  \ord_p\left( \frac{k!}{q(k/q)!} \right)
  &= \frac{k-1}{p-1} - 2 - \frac{k/q-1}{p-1} \\*
  &= (p-1) \left( \frac{k}{q}(p+1) - 2 \right)
  > (p-1)^2 > 0.
\end{align*}
Hence
\[
  A_k[x^k] \equiv 0 \pmod{p}.
\]

We next consider~$C_k[x^k]$. Expanding the power, and noting that
there is an~$x^{q+1}$ that comes out, we find a formula similar to
formula~\eqref{eqn:Bkxkformula} for~$B_k[x^k]$, the primary difference
being that there is no~$1/q$ and the sum is over a different
collection of indices. With the obvious notation for the multinomial
coefficient, we have
\[
  C_k[x^k] = k! \sum_{\substack{e_0+e_1+\cdots+e_{k-1}=q+1\\
      0\cdot e_0+1\cdot e_1+\cdots+(k-1)e_{k-1}=k-q-1\\}}
  \binom{q}{\bfe} \prod_{n=0}^{k-1} 
  \left(\frac{a_n}{n!}\right)^{e_n}.
\]
For any given $\bfe:=(e_0,\ldots,e_{k-1})$, the coefficent of
$\bfa^\bfe:=a_0^{e_0}\cdots a_{k-1}^{e_{k-1}}$ in~$C_k[x^k]$ is
\begin{equation}
  \label{eqn:Ckxkae}
  C_k[x^k][\bfa^\bfe]
  = k! \cdot \binom{q}{e_0,e_1,\ldots,e_{k-1}} \cdot \prod_{n=0}^{k-1} \frac{1}{n!^{e_n}}.
\end{equation}
We claim that~$C_k[x^k][\bfa^\bfe]$ is always divisible by~$p$. To see
this, we compute
\begin{align*}
  (p &- 1)\ord_p C_k[x^k][\bfa^\bfe] \\*
  &= (k-1) + (q-1) - \sum_{n=0}^{k-1} \bigl(e_n-S_p(e_n)\bigr) - \sum_{n=0}^{k-1} e_n\bigl(n-S_p(n)\bigr) \\
  &= (k-1) + (q-1) - (q+1) - (k-q-1) + \sum_{n=0}^{k-1} \bigl(S_p(e_n)+e_nS_p(n)\bigr) \\*
  &=  q-2   + \sum_{n=0}^{k-1} \bigl(S_p(e_n)+e_nS_p(n)\bigr).
\end{align*}
Since $q=p^2\ge4$, this quantity is certainly positive.

Combining our computations for~$A_k[x^k]$,~$B_k[x^k]$, and~$C_k[x^k]$,
we have proven that for~$k$ a power of~$p$, the
recursion~\eqref{eqn:akAkBkCk} for~$a_k$ in terms of~$a_i$ with~$i<k$
has the following form:
\begin{equation}
  \label{eqn:akakpp}
  a_k \equiv -\frac{k!}{q}\cdot \binom{q}{q-p,p}
  \cdot \left(\frac{a_0}{0!}\right)^{q-p}
  \cdot \left(\frac{a_{k/p}}{(k/p)!}\right)^{p}
  \pmod{p}.
\end{equation}
%% It follows  from~\eqref{eqn:ordpBkeq0} and induction
%% that~$a_k$ and~$a_{k/p}$ are not divisible by~$p$.
%% for some $\g_k\in\ZZ_p^*$:
%% \[
%%   a_k \in \g_k a_0^{p^2-p} a_{k/p}^p + p\ZZ_p[a_0,a_1,\ldots,a_{k-1}].
%% \]
%% Since~$a_0=1$ and $a_1=-1$, an easy  induction using this
%% recursion shows taht $a_k\in\ZZ_p^*$ for all~$k$ that are powers of~$p$.

In order to simplify~\eqref{eqn:akakpp}, we set a useful
piece of notation. For $N\ge1$, we let
\[
  T_p(N) := \text{prime-to-$p$ part of $N$} = N/p^{\ord_p(N)}.
\]
Applying~$T_p$ to~\eqref{eqn:akakpp} yields
\[
  \frac{ a_k }{ a_{k/p}^p }
  = -\frac{ T_p(k!) T_p(q!) } { T_p((q-p)!) T_p(p!) T_p((k/p)!)^p },
\]
and since we are only interested in the value of~$T_p(a_k)\bmod p$, we can use
Fermat's little theorem to simplify this to
\begin{equation}
  \label{eqn:TpakTpakpp}
  \frac{ a_k }{ a_{k/p} }
  \equiv -\frac{ T_p(k!) T_p(q!) } { T_p((q-p)!) T_p(p!) T_p((k/p)!) } \pmod{p}.
\end{equation}
It remains to compute the congruence class of the right-hand side
of~\eqref{eqn:TpakTpakpp}, for which we use the following elementary
lemma.

\begin{lemma}
\label{lemma:Tppr}
For all primes $p$ and all $r\ge0$, we have
\[
  T_p\bigl((p^r)!\bigr) \equiv (-1)^r \pmod{p}.
\]
\end{lemma}
\begin{proof}
We may assume that~$p$ is odd.  We group the numbers from~$1$ to~$p^r$
by their~$p$-adic valuation. This yields
\begin{align*}
  T_p\bigl((p^r)!\bigr)
  &= \prod_{j=0}^{r-1} \prod_{\substack{1\le n< p^{r-j}\\ p\nmid n\\}} T_p(p^j n)
  = \prod_{j=0}^{r-1} \prod_{\substack{1\le n< p^{r-j}\\ p\nmid n\\}} n \\
  &\equiv \prod_{j=0}^{r-1} \bigl((p-1)!\bigr)^{p^{r-j-1}}  \pmod{p} \\
  &\equiv \prod_{j=0}^{r-1} (-1)^{p^{r-j-1}}  \pmod{p} \quad\text{Wilson's theorem,} \\
  &\equiv (-1)^r\pmod{p}.
\end{align*}
This completes the proof of Lemma~\ref{lemma:Tppr}.
\end{proof}

Lemma~\ref{lemma:Tppr} allows us to evaluate every factor
in~\eqref{eqn:TpakTpakpp} except $(q-p)!$. For that term, we use the
lemma with~$q=p^2$ and Wilson's theorem to compute
\begin{align}
  \label{eqn:Tpqpfact}
  T_p((q-p)!) 
  &= T_p\left( \frac{ q! }{ q(q-1)\cdots(q-(p-1)) } \right) \notag\\*
  &\equiv \frac{(-1)^2}{(-1)^{p-1}(p-1)!} 
  \equiv -1 \pmod{p}.
\end{align}

Writing $k=p^r$ and using $q=p^2$, we use Lemma~\ref{lemma:Tppr}
and~\eqref{eqn:Tpqpfact} to compute the right-hand side
of~\eqref{eqn:TpakTpakpp}:
\begin{align*}
  \frac{ a_k }{ a_{k/p} }
%%   & -\equiv \frac{ T_p(k!) T_p(q!) } { T_p((q-p)!) T_p(p!) T_p((k/p)!) } \pmod{p}\\
  & \equiv -\frac{ T_p(p^r!) T_p(p^2!) } { T_p((p^2-p)!) T_p(p!) T_p(p^{r-1}!) } \pmod{p}\\
  & \equiv -\frac{ (-1)^r (-1)^2 } { (-1) (-1) (-1)^{r-1} }
  \equiv 1 \pmod{p}.
\end{align*}
This completes the proof that $a_k\equiv a_{k/p}\pmodintext{p}$, and
hence the proof of Theorem~\ref{theorem:radiusp2}.
\end{proof}

\begin{remark}
During the proof of Theorem~\ref{theorem:radiusp2}, we derived a
formula~\eqref{eqn:ordpBka} for the quantity $\ord_p
B_k[x^k][\bfa^\bfe]$ under the assumption that~$q=p^2$ and~$k$ is a
power of~$p$. We note that this calculation easily generalizes for
arbitrary~$q$ and~$k$. Without giving the details, we list the result
in the hope that it might be useful in future investigations, e.g., in
proving Conjecture~\ref{conjecture:xp2akval}:
\begin{multline*}
  (p-1)\ord_p B_k[x^k][\bfa^\bfe] = q - e_0 - (p-1)\ord_p(q) - S_p(k) -
  S_p(q) \\*
   + S_p(e_0) + \sum_{n=1}^{k-1} \bigl( e_n ( S_p(n) - 1 ) +  S_p(e_n) \bigr).
\end{multline*}
\end{remark}

%%%%%%%%%%%%%%%%%%%%%%%%%%%%%%%%%%%%%%%%%%%%%%%%%%%%%%%%%%%%%%%%%%%%%%
\section{The radius of convergence of B\"ottcher coordinates}
\label{section:botradconv}
%%%%%%%%%%%%%%%%%%%%%%%%%%%%%%%%%%%%%%%%%%%%%%%%%%%%%%%%%%%%%%%%%%%%%%

In this section we use
Theorems~\ref{theorem:compm},~\ref{theorem:fphiinRp},
and~\ref{theorem:radiusp2} to prove
Corollary~\ref{corollary:botradconv}, which describes the radius of
convergence of the B\"ottcher coordinate of various sorts of power
series.

We start with an elementary lemma that is undoubtedly well-known, but for
lack of a suitable reference and for the convenience of the reader, we
give the short proof. We set the notation
\[
  \Dcal(R) := \bigl\{ x\in\CC_p : \|x\|_p < R \bigr\}.
\]

\begin{lemma}
\label{lemma:isometry}
\begin{parts}
\Part{(a)}
Let
\[
  f(x) = x\sum_{k=0}^\infty a_kx^k\in\CC_p[\![x]\!]
  \quad\text{with $\|a_0\|_p=1$ and all\/ $\|a_k\|_p\le1$.}
\]
Then~$f(x)$ and its inverse converge and induce
an isometry
\[
  f : \Dcal(1) \xrightarrow{\;\;\sim\;\;} \Dcal(1).
\]
\Part{(b)}
Let
\[
  f(x) = x\sum_{k=0}^\infty \frac{a_kx^k}{k!}\in\CC_p[\![x]\!]
  \quad\text{with $\|a_0\|_p=1$ and all\/ $\|a_k\|_p\le1$.}
\]
Then~$f(x)$ and its inverse converge and induce an isometry
\[
  f : \Dcal(p^{-1/(p-1)}) \xrightarrow{\;\;\sim\;\;} \Dcal(p^{-1/(p-1)}).
\]
\end{parts}
%% and let~$\Dcal$ be the disk
%% \[
%%   \Dcal:=\bigl\{x\in\CC_p : \|x\|_p<p^{-1/(p-1)}\bigr\}.
%% \]
%% \vspace{-10pt}
%% \begin{parts}
%% \Part{(a)}
%% The series~$f(x)$ converges on~$\Dcal$ and induces an isometry
%% $f:\Dcal\to\Dcal$.
%% \Part{(b)}
%% The inverse isometry $f^{-1}:\Dcal\to\Dcal$ is given by a power series
%% \[
%%   f^{-1}(x) = \sum_{k=0}^\infty \frac{b_kx^{k+1}}{k!}\in\CC_p[\![x]\!]
%%   \quad\text{with $\|b_0\|_p=1$ and all $\|b_k\|_p\le1$.}
%% \]
%% \end{parts}  
\end{lemma}
\begin{proof}
We give the proof of~(b), since the proof of~(a) is similar, but
easier.  The fact that~$f(x)$ converges on~$\Dcal(p^{-1/(p-1)})$ is
standard and follows from the estimate
\begin{multline*}
  \limsup_{k\to\infty} \left\| \frac{a_k}{k!} \right\|_p^{-\frac{1}{k+1}}
  \ge
  \limsup_{k\to\infty} \|k!\|_p^{-\frac{1}{k+1}} \\
  =
  \smash[t]{
    \limsup_{k\to\infty} \left(p^{-\frac{k-S_p(k)}{p-1}}\right)^{-\frac{1}{k+1}}
  =
  p^{-\frac{1}{p-1}}.}
\end{multline*}
For the last equality, we used $S_p(k)\le(p-1)\log(k)/\log(p)$, valid
for all $k\ge1$.

In order to show that~$f$ is an isometry, we note that
\[
  \frac{f(x)-f(y)}{x-y}
  =
  a_0 + \sum_{k=1}^\infty \frac{a_k}{k!}\cdot \frac{x^{k+1}-y^{k+1}}{x-y}.
\]
Since $\|a_0\|_p=1$ by assumption, we need to show that every term in
the sum has norm strictly smaller than~$1$.  We first observe that for
$k\ge0$ and (distinct) $x,y\in\Dcal$, we have
\begin{equation}
  \label{eqn:absxkykxy}
  \left\| \frac{x^{k+1}-y^{k+1}}{x-y} \right\|_p
  = \left\| \sum_{i=0}^k x^iy^{k-i} \right\|_p
  \le \max_{0\le i\le k} \|x^iy^{k-i}\|_p
  \le   p^{-\frac{k}{p-1}}.
\end{equation}
Hence using $\|a_k\|_p\le1$ and the fact that~$S_p(k)\ge1$ for
all~$k\ge1$, we find that
\begin{align*}
  \sup_{k\ge1} \left\| \frac{a_k}{k!}\cdot \frac{x^{k+1}-y^{k+1}}{x-y} \right\|_p
  &\le 
  \sup_{k\ge1} \left\| \frac{1}{k!}\right\|_p \cdot p^{-\frac{k}{p-1}}
  \quad\text{from \eqref{eqn:absxkykxy},} \\*
  &= 
  \sup_{k\ge1} p^{\frac{k-S_p(k)}{p-1}}\cdot p^{-\frac{k}{p-1}} \\*
  &=
  \sup_{k\ge1} p^{-\frac{S_p(k)}{p-1}}
  =
  p^{-\frac{1}{p-1}}
  <
  1.
\end{align*}
This concludes the proof of~(b).
\end{proof}

%% *** There must be some standard sort of theorem saying that if
%% $f:D(r)\to D(r)$ is a power series inducing an isometry, the~$f^{-1}$
%% is given by a power series converging on $D(r)$.

%% RobB says the following stronger result is known and quite standard now.
%% (He suggests it may be in Hsia's thesis, for example.)

%% If $f : D \to f(D)$ be injective (equivalently, bijective,
%% equivalently, have Weierstrass degree~$1$), then~$f^{-1}:f(D)\to D$ is
%% given by a convergent power series. This is Proposition 3.4.7 in my
%% old version of Rob's book (although he says it's Prop. 3.4.8 in the
%% current version), and the proof is in Section 12.4 (page 283).
%% That proposition assumes that $f$ is bijective. I asked Rob if
%% injective implies isometric, or if isometric is strictly stronger.

\begin{proof}[Proof of Corollary~$\ref{corollary:botradconv}$]
Proposition~\ref{proposition:botmprimetop} and Lemma~\ref{lemma:isometry}(a) immediately
imply the part of Corollary~\ref{corollary:botradconv}(a) with $p\nmid m$,
while Theorem~\ref{theorem:fphiinRp} and Lemma~\ref{lemma:isometry}(a) immediately
imply Corollary~\ref{corollary:botradconv}(c).

Similarly, Theorem~\ref{theorem:compm}(b) and
Lemma~\ref{lemma:isometry}(b) immediately imply
Corollary~\ref{corollary:botradconv}(b).

Next, to prove the~$p\mid m$ part of
Corollary~\ref{corollary:botradconv}(a), we use
Theorem~\ref{theorem:compm}(a) and apply Lemma~\ref{lemma:isometry}(b)
to the series~$mf_\f(x/m)$. Undoing this transformation transforms the
disk $\Dcal(p^{-1/(p-1)})$ to the disk
$\Dcal\bigl(p^{-1/(p-1)}\|m\|_p\bigr)$.

Finally, to prove~(d), we use Theorem~\ref{theorem:radiusp2} to compute
the radius of convergence of~$f_\f(x)$. Thus
\begin{align*}
  \rho(f_\f) &:= \liminf_{k\to\infty} \|a_k/k!\|_p^{-1/(k+1)} \\
  & = \liminf_{k\to\infty} p^{\frac{\ord_p(a_k) - \ord_p(k!)}{k+1}} \\*
  & = \liminf_{k\to\infty} p^{\left(\ord_p(a_k)-\frac{k-S_p(k)}{p-1}\right)\cdot\frac{1}{k+1}} \\
  & = p^{-1/(p-1)} \cdot  \liminf_{k\to\infty} p^{\ord_p(a_k)/(k+1)} \\*
  &\omit\hfill\qquad since $S_p(k)\le(p-1)\log_p(k+1)$, so $S_p(k)/(k+1)\to0$,\\
  & = p^{-1/(p-1)} \\*
  &\omit\hfill\qquad since Theorem~\ref{theorem:radiusp2} says $a_k=-1$ for infinitely many~$k$.
\end{align*}
This completes the proof of Corollary~\ref{corollary:botradconv}.
\end{proof}

\section{The B\"ottcher coordinate of $x^{p^2}+p^rx^{p^2+1}$}
\label{section:xp2conj}
%%%%%%%%%%%%%%%%%%%%%%%%%%%%%%%%%%%%%%%%%%%%%%%%%%%%%%%%%%%%%%%%%%%%%%
This section describes a number of conjectures prompted by moderately
extensive numerical experiments.

Theorem~\ref{theorem:radiusp2} says that the $k$'th
coefficient~$a_k/k!$ of the B\"ottcher coordinate
of~$x^{p^2}+p^2x^{p^2+1}$ satisfies $a_k\equiv-1\pmodintext{p}$
whenever~$k$ is a power of~$p$.  Experiments suggest that the sequence
of mod~$p$ values of the~$a_k$ have many interesting properties.
Computations for primes $2\le p\le11$ and $0\le k\le50$ lead us to
make the following conjecture.

\begin{conjecture}
\label{conjecture:xp2akval}
Let $\f(x)=x^{p^2}+p^2 x^{p^2+1}$, and write the B\"ottcher coordinate
of~$\f$ as $f_\f(x)=x\sum_{k=0}^\infty (a_k/k!)x^k$, so
Theorem~$\ref{theorem:radiusp2}(b)$ tells us that the~$a_k$ are
integers. Then
\begin{align*}
  k\equiv 0\pmodintext{p}
  &\quad\Longrightarrow\quad
  a_k \equiv (-1)^{k/p} \pmodintext{p}, \\*
  %%%%%%%%%%%%%%%%%%%%%%%%%%%%%%
  k\equiv -2\pmodintext{p}
  &\quad\Longrightarrow\quad
  a_k \equiv -1 \pmodintext{p}, \\*
  %%%%%%%%%%%%%%%%%%%%%%%%%%%%%%
  k\equiv -1\pmodintext{p}
  &\quad\Longleftrightarrow\quad
  a_k \equiv 0 \pmodintext{p}, 
\end{align*}
with the one exception that for~$p=2$, we have $a_1\equiv1\pmodintext{2}$. 
%% PARI code: vvpn(p, n, s, vv) = vv=Vec(Seriesf(p^2,n,p^2));vv=vector(#vv,k,vv[k]*(k-1)!)%p;vector(floor((#vv-s-1)/p)+1,k,vv[p*(k-1)+1+s])
%% *** And something similar(?) is true for, e.g., $\f(x)=x^6+2x^6$. 
\end{conjecture}

\begin{remark}
Let~$B_k[x^k]$ be as in the proof of
Theorem~\ref{theorem:radiusp2}.  The key to the proof of that
theorem is that when~$k$ is a power of~$p$, the expression
for~$B_k[x^k]$ as a polynomial in~$a_0,\ldots,a_{k-1}$ contains a
single monomial whose coefficient is not divisible by~$p$. Precisely,
we proved that
\[
  - B_k[x^k] \equiv  a_0^{p^2-p} a_{k/p}^p \pmodintext{p}.
\]
One approach to the first part of Conjecture~\ref{conjecture:xp2akval}
might be to show that $B_k[x^k]\bmod p$ consists of a single monomial
under the weaker assumption that~$p\mid k$. Unfortunately, this is
not true. To illustrate, we list the first few cases for which it fails for~$p=2$:
\begin{align*}
B_{14}[x^{14}]\bmod2 &=  a_0^2 (a_2 a_{12}  +   a_4 a_{10} +  a_6 a_8),  \\*
B_{22}[x^{22}]\bmod2 &=  a_0^2 ( a_2 a_{20}  +  a_4 a_{18}  +   a_6 a_{16}),  \\
B_{26}[x^{26}]\bmod2 &=  a_0^2 (a_2 a_{24}  +   a_8 a_{18} +  a_{10} a_{16}),  \\
B_{28}[x^{28}]\bmod2 &=  a_0^2 (a_4 a_{24}  +  a_8 a_{20}  + a_{12} a_{16} ),  \\*
B_{30}[x^{30}]\bmod2 &=  a_0^2(  a_2 a_{28} + a_4 a_{26}  +  a_6 a_{24} +  a_8 a_{22}  \\*
&\omit\hfill${}+  a_{10} a_{20} +  a_{12} a_{18}  +  a_{14} a_{16} )$ . 
\end{align*}
We observe that in each case the sum consists of an odd number of
monomials, each of which is a product of even index~$a_i$, which is
consistent with the conjecture that~$B_k[x^k]$ is  odd when~$k$ is
even.
\end{remark}

%% \begin{align*}
%% B_{2}[x^{2}]\bmod2 &=  a_0^2 a_1^2  \\
%% B_{4}[x^{4}]\bmod2 &=  a_0^2 a_2^2  \\
%% B_{6}[x^{6}]\bmod2 &=  a_0^2  a_2  a_4 \\
%% B_{8}[x^{8}]\bmod2 &=  a_0^2 a_4^2  \\
%% B_{10}[x^{10}]\bmod2 &=  a_0^2  a_2  a_8 \\
%% B_{12}[x^{12}]\bmod2 &=  a_0^2  a_4  a_8 \\
%% B_{14}[x^{14}]\bmod2 &=  a_0^2 (a_2 a_{12}  +   a_4 a_{10} +  a_6 a_8)  \\
%% B_{16}[x^{16}]\bmod2 &=  a_0^2 a_8^2 \\
%% B_{18}[x^{18}]\bmod2 &=  a_0^2  a_2 a_{16}  \\
%% B_{20}[x^{20}]\bmod2 &=  a_0^2  a_4 a_{16}  \\
%% B_{22}[x^{22}]\bmod2 &=  a_0^2 ( a_2 a_{20}  +  a_4 a_{18}  +   a_6 a_{16})  \\
%% B_{24}[x^{24}]\bmod2 &=  a_0^2 a_8 a_{16}   \\
%% B_{26}[x^{26}]\bmod2 &=  a_0^2 (a_2 a_{24}  +   a_8 a_{18} +  a_{10} a_{16})  \\
%% B_{28}[x^{28}]\bmod2 &=  a_0^2 (a_4 a_{24}  +  a_8 a_{20}  + a_{12} a_{16} )  \\
%% B_{30}[x^{30}]\bmod2 &=  a_0^2(  a_2 a_{28} + a_4 a_{26}  +  a_6 a_{24} +  a_8 a_{22}  +  a_{10} a_{20}  +  a_{12} a_{18} 
%%    +  a_{14} a_{16} )  \\ 
%% B_{32}[x^{32}]\bmod2 &=  a_0^2  a_{16}^2 
%% \end{align*}

%% For example, if $k\equiv2\pmodintext4$, we claim that the coefficient
%% of~$a_0^2a_2a_{k-2}$ is odd. To prove this, we observe that 
%% \[
%%   B_k[x^k][a_0^2a_2a_{k-2}]
%%   = \frac{k!}{4}\binom{4}{2,1,1}\frac{1}{(0!)^2\cdot (2!)^1\cdot (k-2)!}
%%   = \frac{3\cdot k\cdot (k-1)}{2}
%%   \equiv 1\pmodintext{2},
%% \]
%% where the final congruence uses our assumption that $k\equiv2\pmodintext4$.

We next consider the coefficients for the B\"ottcher coordinate of
\[
  \f(x) = x^{p^2} + p^r x^{p^2+1}
\]
as~$r$ increases. Not surprisingly, the integrality of the
coefficients of the B\"ottcher coordinate increases with
increasing~$r$, but we conjecture that we never achieve a B\"ottcher
coordinate whose coefficients are entirely~$p$-integral.

More precisely, experiments suggest that the valuations of the
B\"ott\-cher quantities~$a_k$ for~$p\mid k$ exhibit a great deal of
regularity.  Here is a typical example. For the function
\[
  \f(x) = x^4 + 2^8 x^5,
\]
we observe numerically that that the coefficients of the B\"ottcher coordinate $f_\f(x)=x\sum
(a_k/k!)x^k$ satisfy the following somewhat complicated recursion:
\begin{multline*}
  \ord_2(a_{2n+2})
  - \ord_2(a_{2n}) \\
  = \begin{cases}
     12 &\text{if $n\equiv0\pmodintext2$,} \\
     - 5 &\text{if $n\equiv1\pmodintext2$ and $n\not\equiv-1\pmodintext8$,} \\
     - 16 &\text{if $n\equiv-1\pmodintext8$ and $n\not\equiv-1\pmodintext{32}$.} \\
     - 21 &\text{if $n\equiv-1\pmodintext{32}$.} \\
  \end{cases}
\end{multline*}
More generally, for
\[
%% \f(x) = x^4 + 4c x^5\quad\text{with $c=2^r$ a power of~$2$,}
  \f(x) = x^4 + 2^{2+r} x^5\quad\text{with $r\ge0$,}
\]
we find (numerically) that there are recursions for even values of~$k$
having the form
\begin{align*}
  \ord_2(a_{k+2^{r+1}}) &= \ord_2(a_{k}) + 2^{r+1}-2 &&\text{if $r$ is odd,} \\*
  \ord_2(a_{k+2^r}) &= \ord_2(a_{k}) + 2^r-1 &&\text{if $r$ is even.} 
\end{align*}

Based on these and other computations, we make the following conjecture.

\begin{conjecture}
Let
\[
  \f(x) = x^{p^2} + p^{2+r} x^{p^2+1}\quad\text{with $r\ge0$,}
\]
and let $f_\f(x)=x\sum (a_k/k!)x^k$ be the associated B\"ottcher
coordinate.%
\begin{parts}
  \Part{(a)}
  For $k\equiv0\pmod{p}$ we have
  \[
  \ord_p(a_k) = \frac{1-p^{-r}}{p-1}k+O(1)\quad\text{as $k\to\infty$.}
  \]
  \Part{(b)}
  The radius of convergence of~$f_\f$ is
  \[
  \rho(f_\f) := \liminf_{k\to\infty} \|a_k/k!\|_p^{-1/k} = p^{-p^{-r}/(p-1)}.
  \]
\end{parts}
\end{conjecture}

We remark that~(b) does not follow from~(a). Indeed, applying~(a)
with~$k$ a power of~$p$ only shows that the radius of convergence is
bounded above by the quantity given by~(b).

%%%%%%%%%%%%%%%%%%%%%%%%%%%%%%%%%%%%%%%%%%%%%%%%%%%%%%%%%%%%%%%%%%%%%%%%
% Acknowldegements and Bibliography
%%%%%%%%%%%%%%%%%%%%%%%%%%%%%%%%%%%%%%%%%%%%%%%%%%%%%%%%%%%%%%%%%%%%%%%%

\begin{acknowledgement}
The authors would like to thank Rob Benedetto, Patrick Ingram, Holly
Krieger, Jonathan Lubin, and Mike Zieve for their helpful advice, and
Melody Chan for her insight into the proof of Lemma~\ref{lemma:combin}.
\end{acknowledgement}

%% \begin{thebibliography}{99}
%% \itemsep=\smallskipamount
%% \end{thebibliography}

%% \bibliographystyle{plain}
%% \bibliography{/Users/jhs/Dropbox/AAJHS/Book/ADS/ArithDyn,padic}

\def\cprime{$'$}

\end{document}